\documentclass{siamltex}
\usepackage{graphics}
\usepackage{amsfonts}
\usepackage{amssymb}
\usepackage{mathtools}
\usepackage{color}
\usepackage{enumitem}
\newtheorem{remark}[theorem]{Remark}

\newcommand{\R}{{\mathbb R}}
\newcommand{\C}{{\mathbb C}}

\begin{document}
\title{Eigenvector sensitivity under general and structured perturbations 
of tridiagonal Toeplitz-type matrices}

\author{Silvia Noschese\thanks{Dipartimento di Matematica ``Guido Castelnuovo'', SAPIENZA 
Universit\`a di Roma, P.le A. Moro, 2, I-00185 Roma, Italy. E-mail:
{\tt noschese@mat.uniroma1.it}. Research supported by a grant from
SAPIENZA Universit\`a di Roma. The author is a member of the INdAM Research Group GNCS.}
\and
Lothar Reichel\thanks{Department of Mathematical Sciences, Kent State University, Kent, OH
44242, USA. E-mail: {\tt reichel@math.kent.edu}. Research supported in part by NSF grants
DMS-1729509 and DMS-1720259.}
}

\maketitle

\begin{abstract}
The sensitivity of eigenvalues of structured matrices under general or structured 
perturbations of the matrix entries has been thoroughly studied in the literature. Error 
bounds are available and the pseudospectrum can be computed to gain insight. Few
investigations have focused on analyzing the sensitivity of eigenvectors under general
or structured perturbations. The present paper discusses this sensitivity for 
tridiagonal Toeplitz and Toeplitz-type matrices. 
\end{abstract}

\maketitle

\section{Introduction}\label{sec1}
The sensitivity of the eigenvalues of a structured matrix to general or structured 
perturbations of the matrix entries has received considerable attention in the literature.
Both bounds and graphical tools such as the pseudospectrum or structured pseudospectrum 
have been developed; see, e.g., \cite{BG05,BGN,NPR,NR19,TE,Wi}. While the pseudospectrum measures 
the sensitivity of the eigenvalues, it depends on the sensitivity of the eigenvectors of 
the matrix to perturbations of the matrix entries. However, we are not aware of 
investigations that focus on the sensitivity of the eigenvectors to general or structured
perturbations of a structured matrix. It is the purpose of the present paper to carry out 
such an investigation for tridiagonal Toeplitz matrices and tridiagonal Toeplitz-type 
matrices that are obtained by modifying the first and last diagonal entries of a 
tridiagonal Toeplitz matrix. These kinds of matrices arise in numerous applications, 
including the solution of ordinary and partial differential equations \cite{DL98,FGHLW74,Sm,YC08}, 
time series analysis \cite{LP10}, and as regularization matrices in Tikhonov regularization for 
the solution of discrete ill-posed problems \cite{H98,RY09}. It is therefore important to 
understand properties of these matrices relevant for computation. Our analysis is 
facilitated by the fact that the eigenvalues and eigenvectors of the matrices considered 
are known in closed form. 

Introduce the tridiagonal Toeplitz matrix 
\begin{equation}\label{TRap}
T=\left[ 
\begin{array}{ccccccc}
\delta & \tau &  &  &  &  & \Large{O} \\ 
\sigma & \delta & \tau &  &  &  &  \\ 
& \sigma & \cdot & \cdot &  &  &  \\ 
&  & \cdot & \cdot & \cdot &  &  \\ 
&  &  & \cdot & \cdot & \cdot &  \\ 
&  &  &  & \cdot & \cdot & \tau \\ 
\Large{O} &  &  &  &  & \sigma & \delta
\end{array}
\right] \,\in {\mathbb C}^{n\times n}.
\end{equation}
We will denote this matrix by $T=(n;\sigma ,\delta ,\tau)$. It is well known that its 
eigenvalues are given by
\begin{equation}\label{lamhT}
\lambda_{h}=\delta +2\sqrt{\sigma\tau }\,\cos \frac{h\pi}{n+1}, \quad h=1:n;
\end{equation}
see, e.g., \cite{Sm}. Assume that $\sigma \tau \neq 0$. Then the matrix (\ref{TRap}) has
$n$ simple eigenvalues, which lie on a line segment that is symmetric with respect 
to $\delta$. The components of the right eigenvector 
$x_{h}=[x_{h,1},x_{h,2},\ldots,x_{h,n}]^T\in{\mathbb C}^n$, for $h=1:n$, associated with 
the eigenvalue $\lambda_{h}$ are given by 
\begin{equation}\label{xhk}
x_{h,k}=\left(\sqrt{\frac{\sigma} {\tau }}\right)^{k}\sin \frac{hk\pi}{n+1},\quad k=1:n,
\end{equation}
and the corresponding left eigenvector 
$y_{h}=[y_{h,1},y_{h,2},\ldots,y_{h,n}]^T\in{\mathbb C}^n$ has the components 
\begin{equation}\label{yhk}
y_{h,k}=\left(\sqrt{\frac{\bar{\tau}}{\bar{\sigma}}}\right)^{k}\sin \frac{hk\pi}{n+1},
\quad k=1:n,
\end{equation}
where the bar denotes complex conjugation. Throughout this paper the superscript 
$(\cdot)^T$ stands for transposition and the superscript $(\cdot)^H$ for transposition and
complex conjugation.

If $\sigma=0$ and $\tau\neq 0$ (or $\sigma\neq 0$ and $\tau=0$), then the matrix 
(\ref{TRap}) has the unique eigenvalue $\delta$ of geometric multiplicity one. The right 
and left eigenvectors are the first and last columns (or the last and first columns) of 
the identity matrix, respectively.

We also will consider tridiagonal Toeplitz-type matrices of the form 
\begin{equation}\label{Talbeta}
T_{\alpha,\beta}=\left[ 
\begin{array}{ccccccc}
\delta-\alpha & \tau &  &  &  &  & \Large{O} \\ 
\sigma & \delta & \tau &  &  &  &  \\ 
& \sigma & \cdot & \cdot &  &  &  \\ 
&  & \cdot & \cdot & \cdot &  &  \\ 
&  &  & \cdot & \cdot & \cdot &  \\ 
&  &  &  & \cdot & \cdot & \tau \\ 
\Large{O} &  &  &  &  & \sigma & \delta-\beta
\end{array}
\right] \,\in {\mathbb C}^{n\times n}
\end{equation}
for certain parameters $\alpha,\beta\in{\mathbb C}$.
These matrices arise in the solution of ordinary or partial differential equations on an
interval. Thus, $T_{\alpha,\beta}$ is a tridiagonal Toeplitz matrix when $\alpha=\beta=0$.

Formulas for eigenvalues and eigenvectors of the matrices (\ref{Talbeta}) are explicitly
known for several choices of the parameters $\alpha$ and $\beta$; they are derived in 
\cite{Y}. Table \ref{table1} reports expressions for the eigenvalues for several choices 
of $\alpha$ and $\beta$. When $\sigma\tau\neq 0$, the components of the right eigenvector 
$x_h=[x_{h,1},x_{h,2},\ldots,x_{h,n}]^T\in{\mathbb C}^n$ associated with the eigenvalue 
$\lambda_{h}$ are given by
\[
\begin{array}{rclll}
x_{h,k}&=&\left(\sqrt{\frac{\sigma}{\tau }}\right)^{k}\sin\frac{2hk\pi}{2n+1},&\quad 
\alpha=0,&\quad\beta=\sqrt{\sigma\tau};\\
x_{h,k}&=&\left(\sqrt{\frac{\sigma}{\tau }}\right)^{k}\sin\frac{h(2k-1)\pi}{2n+1},&\quad
\alpha=\sqrt{\sigma\tau },&\quad\beta=0;\\
x_{h,k}&=&\left(\sqrt{\frac{\sigma}{\tau }}\right)^{k}\sin\frac{(2h-1)k\pi}{2n+1},&\quad
\alpha=0,&\quad\beta=-\sqrt{\sigma\tau};\\
x_{h,k}&=&\left(\sqrt{\frac{\sigma}{\tau}}\right)^{k}\cos\frac{(2h-1)(2k-1)\pi}{2(2n+1)},&
\quad \alpha=-\sqrt{\sigma\tau },&\quad\beta=0;\\
x_{h,k}&=&\left(\sqrt{\frac{\sigma}{\tau}}\right)^{k}\sin\frac{(2h-1)(2k-1)\pi}{4n},&
\quad \alpha=\sqrt{\sigma\tau },&\quad\beta=-\sqrt{\sigma\tau};\\
x_{h,k}&=&\left(\sqrt{\frac{\sigma}{\tau}}\right)^{k}\cos\frac{(2h-1)(2k-1)\pi}{4n},&\quad
\alpha=-\sqrt{\sigma\tau },&\quad\beta=\sqrt{\sigma\tau};\\
x_{h,k}&=&\left(\sqrt{\frac{\sigma}{\tau}}\right)^{k}\sin\frac{h(2k-1)\pi}{2n},&\quad 
\alpha=\sqrt{\sigma\tau},&\quad\beta=\sqrt{\sigma\tau};\\
x_{h,k}&=&\left(\sqrt{\frac{\sigma}{\tau}}\right)^{k}\cos\frac{(h-1)(2k-1)\pi}{2n},&\quad
\alpha=-\sqrt{\sigma\tau },&\quad\beta=-\sqrt{\sigma\tau},\\
\end{array}
\]
for $k=1:n$.

\begin{table}
\begin{center}
    \begin{tabular}{c|c|c}
    $\alpha$ & $\beta$& $\lambda_h$  \\ \hline 
    $0$&$\sqrt{\sigma\tau }$&$\delta +2\sqrt{\sigma\tau }\,\cos \frac{2h\pi }{2n+1}$\\
    $\sqrt{\sigma\tau }$&$0$&$\delta +2\sqrt{\sigma\tau }\,\cos \frac{2h\pi }{2n+1}$\\
    $0$&$-\sqrt{\sigma\tau }$&$\delta +2\sqrt{\sigma\tau }\,\cos \frac{(2h-1)\pi }{2n+1}$\\
    $-\sqrt{\sigma\tau }$&$0$&$\delta +2\sqrt{\sigma\tau }\,\cos \frac{(2h-1)\pi }{2n+1}$\\
    $\sqrt{\sigma\tau }$&$-\sqrt{\sigma\tau }$&$\delta +2\sqrt{\sigma\tau }\,\cos \frac{(2h-1)\pi }{2n}$\\
    $-\sqrt{\sigma\tau }$&$\sqrt{\sigma\tau }$&$\delta +2\sqrt{\sigma\tau }\,\cos \frac{(2h-1)\pi }{2n}$\\
    $\sqrt{\sigma\tau }$&$\sqrt{\sigma\tau }$&$\delta +2\sqrt{\sigma\tau }\,\cos \frac{h\pi }{n}$\\
    $-\sqrt{\sigma\tau }$&$-\sqrt{\sigma\tau }$&$\delta +2\sqrt{\sigma\tau }\,\cos \frac{(h-1)\pi }{n}$
        \end{tabular}
\end{center}
\caption{Formulas for the eigenvalues $\lambda_h$ of the matrix (\ref{Talbeta}) for
$h=1:n$ and several choices of $\alpha$ and $\beta$.}
\label{table1}
\end{table}

It is straightforward to show that the component $y_{h,k}$ of the left eigenvector
$y_h=[y_{h,1},\ldots,y_{h,n}]^T\in{\mathbb C}^n$ of \eqref{Talbeta} is obtained from the
component $x_{h,k}$ of the corresponding right eigenvector by replacing the factor
$(\sigma/\tau)^{k/2}$ by $(\bar{\tau}/\bar{\sigma})^{k/2}$.

This paper is organized as follows. Section \ref{sec2} discusses the sensitivity of the
eigenvalues of the matrices (\ref{TRap}) and (\ref{Talbeta}) to general (unstructured) 
perturbations. Eigenvalue condition numbers for the matrices (\ref{TRap}) and 
(\ref{Talbeta}) are given. Section \ref{sec3} is concerned with the sensitivity of the
eigenvectors of the matrices (\ref{TRap}) and (\ref{Talbeta}) to general perturbations. 
Eigenvector condition numbers are presented. Section \ref{sec4} discusses eigenvalue and 
eigenvector sensitivity to structured perturbations. Condition numbers are defined. 
Section \ref{sec5} describes two novel applications of tridiagonal Toeplitz matrices. The
first part of the section shows how eigenvalues of a symmetric tridiagonal matrix can 
be estimated by using the explicitly known eigenvalues of the closest symmetric 
tridiagonal Toeplitz matrix. In the latter part of the section, we discuss how the 
eigenvectors of a severely nonsymmetric nearly tridiagonal Toeplitz matrix can be 
computed accurately by using the explicitly known spectral factorization of the closest 
tridiagonal Toeplitz matrix. Concluding remarks can be found in Section \ref{sec6}.

\section{Sensitivity of the spectrum}\label{sec2}
This section discusses the sensitivity of the eigenvalues of the matrices (\ref{TRap}) and
(\ref{Talbeta}) to general (unstructured) perturbations of the matrix entries.

\subsection{Eigenvalue distances}\hfill\break

\begin{proposition}\label{prop1}
The eigenvalues (\ref{lamhT}) of the matrix $T$ defined by (\ref{TRap}) satisfy
\begin{equation}\label{evaldist}
\min_{\lambda_j \neq\lambda_h }|\lambda_h -\lambda_j |= 
\begin{cases}
        4\sqrt{|\sigma\tau|}\sin \frac{\pi }{2(n+1)}\sin \frac{(2h-1)\pi }{2(n+1)}, & \text{for }   1< h\leq \frac n 2 \text{  or } h=n,\\
        \\
        4\sqrt{|\sigma\tau|}\sin \frac{\pi }{2(n+1)}\sin \frac{(2h+1)\pi }{2(n+1)}, & \text{for }  h=1 \text{  or } \frac n 2< h<n.\\
\end{cases}
\end{equation}
In particular, the distance of the eigenvalue $\lambda_{h}$ to the other eigenvalues of 
$T$ only depends on $h$, $n$, and the product $|\sigma\tau|$. Moreover, the minimal 
distance between any two eigenvalues of $T$ is 
\[
4\sqrt{|\sigma\tau|}\sin \frac{\pi }{2(n+1)}\sin \frac{3\pi }{2(n+1)}.
\]
This distance is achieved by $|\lambda_1-\lambda_2|$ and $|\lambda_{n-1}-\lambda_{n}|$.
\end{proposition}

\begin{proof}
Let $1\leq j,h\leq n$. The trigonometric identity
\[
\cos a-\cos b =-2\sin \frac{a+b}{2}\sin \frac{a-b}{2}
\]
yields
\begin{eqnarray*}
\min_{\lambda_j \neq\lambda_h }|\lambda_h -\lambda_j |&=&
\min \left\{|\lambda_h -\lambda_{h+1} |,|\lambda_h -\lambda_{h-1} | \right\}\\
&=&2\sqrt{|\sigma\tau|}\min \left\{\left|\cos \frac{h\pi }{n+1}-\cos \frac{(h+1)\pi }{n+1}\right|,
\left|\cos \frac{h\pi }{n+1}-\cos \frac{(h-1)\pi}{n+1}\right|\right\} \\
&=&4\sqrt{|\sigma\tau|}\sin\frac{\pi }{2(n+1)}\min\left\{\left|\sin\frac{(2h+1)\pi}{2(n+1)}\right|,
\left|\sin\frac{(2h-1)\pi}{2(n+1)}\right|\right\}. 
\end{eqnarray*}
This shows (\ref{evaldist}). The remaining statements follow from this formula.
\end{proof}

\begin{remark}
Results on the spacing of the eigenvalues of Hermitian  Toeplitz matrices with simple-loop 
symbols (e.g., of the eigenvalues of Hermitian tridiagonal Toeplitz matrices) are reported
in \cite{BBGM15,BBGM17}. Such results can be extended to non-Hermitian tridiagonal 
Toeplitz matrices by diagonal similarity transformation.  To this end, we note that the
matrix $T=(n;\sigma,\delta,\tau)$ is, via the diagonal matrix $D=\diag[1,v,\ldots,v^{n-1}]$,
similar to $T'=(n;v\sigma,\delta,v^{-1}\tau)$. One can choose $v$ so that 
$|v \sigma| = |v^{-1} \tau|$. The matrix $T'$ then is normal; see \cite[Theorem 3.1]{NPR}. 
In particular, real tridiagonal Toeplitz matrices $T$ can be transformed to symmetric 
matrices $T'$ by letting $v$ be such that $v \sigma = v^{-1} \tau$.
\end{remark}

An analogue of Proposition \ref{prop1} can be shown for the eigenvalues of the matrix
(\ref{Talbeta}) for the choices of $\alpha$ and $\beta$ considered in Table \ref{table1}.
The results follow from the expressions for the eigenvalues in this table and are 
formulated in the following proposition.

\begin{proposition}\label{prop1bis}
The eigenvalues $\lambda_h$ of the matrix $T_{\alpha,\beta}$ satisfy
\begin{enumerate}[label=(\roman*)]
\item for $\alpha=0$ and $ \beta=\sqrt{\sigma\tau }$ or vice versa
\begin{equation*}
\min_{\lambda_j \neq\lambda_h }|\lambda_h -\lambda_j |= 
\begin{cases}
  4\sqrt{|\sigma\tau|}\sin \frac{\pi }{2n+1}\sin \frac{(2h-1)\pi }{2n+1}, & \text{for } 
  1< h\leq \frac n 2 \text{  or } h=n,\\
        \\
  4\sqrt{|\sigma\tau|}\sin \frac{\pi }{2n+1}\sin \frac{(2h+1)\pi }{2n+1}, & \text{for } 
  h=1 \text{  or } \frac n 2< h<n;\\
\end{cases}
\end{equation*}

\item for $\alpha=0$ and $ \beta=-\sqrt{\sigma\tau }$ or vice versa
\begin{equation*}
\min_{\lambda_j \neq\lambda_h }|\lambda_h -\lambda_j |= 
\begin{cases}
  4\sqrt{|\sigma\tau|}\sin \frac{\pi }{2n+1}\sin \frac{2(h-1)\pi }{2n+1}, & \text{for } 
  1< h\leq \lceil \frac n 2 \rceil \text{  or } h=n,\\
        \\
  4\sqrt{|\sigma\tau|}\sin \frac{\pi }{2n+1}\sin \frac{2h\pi }{2n+1}, & \text{for } 
  h=1 \text{  or } \lceil \frac n 2 \rceil< h<n;\\
\end{cases}
\end{equation*}

\item for $\alpha=\sqrt{\sigma\tau }$ and $ \beta=-\sqrt{\sigma\tau }$ or vice versa

\begin{equation*}
\min_{\lambda_j \neq\lambda_h }|\lambda_h -\lambda_j |= 
\begin{cases}
  4\sqrt{|\sigma\tau|}\sin \frac{\pi }{n}\sin \frac{(h-1)\pi }{n}, & \text{for } 
  1< h\leq \frac n 2 \text{  or } h=n,\\
        \\
  4\sqrt{|\sigma\tau|}\sin \frac{\pi }{n}\sin \frac{h\pi }{n}, & \text{for } 
  h=1 \text{  or } \frac n 2< h<n;\\
\end{cases}
\end{equation*}

\item for $\alpha=\sqrt{\sigma\tau }$ and $ \beta=\sqrt{\sigma\tau }$

\begin{equation*}
\min_{\lambda_j \neq\lambda_h }|\lambda_h -\lambda_j |= 
\begin{cases}
        4\sqrt{|\sigma\tau|}\sin \frac{\pi }{2n}\sin \frac{(2h-1)\pi }{2n}, & \text{for }   1< h\leq \frac n 2 \text{  or } h=n,\\
        \\
        4\sqrt{|\sigma\tau|}\sin \frac{\pi }{2n}\sin \frac{(2h+1)\pi }{2n}, & \text{for }  h=1 \text{  or } \frac n 2< h<n;\\
\end{cases}
\end{equation*}
\item for $\alpha=-\sqrt{\sigma\tau }$ and $ \beta=-\sqrt{\sigma\tau }$

\begin{equation*}
\min_{\lambda_j \neq\lambda_h }|\lambda_h -\lambda_j |= 
\begin{cases}
   4\sqrt{|\sigma\tau|}\sin \frac{\pi }{2n}\sin \frac{(2h-3)\pi }{2n}, & \text{for } 
   1< h\leq \lceil \frac n 2 \rceil \text{  or } h=n,\\
        \\
   4\sqrt{|\sigma\tau|}\sin \frac{\pi }{2n}\sin \frac{(2h-1)\pi }{2n}, & \text{for } 
   h=1 \text{  or } \lceil \frac n 2 \rceil< h<n.\\
\end{cases}
\end{equation*}
\end{enumerate}
\end{proposition}

Table \ref{table2} shows the minimal distance between any two eigenvalue of 
$T_{\alpha,\beta}$ for the choices of $\alpha$ and $\beta$ of Table \ref{table1}.

\begin{table}
\begin{center}
\begin{tabular}{c|c|c}
    $\alpha$ & $\beta$& Minimal distance  \\
    \hline 
    $0$&$\sqrt{\sigma\tau }$&$4\sqrt{|\sigma\tau|}\sin \frac{\pi }{2n+1}\sin \frac{2\pi }{2n+1}$\\
    $\sqrt{\sigma\tau }$&$0$&$4\sqrt{|\sigma\tau|}\sin \frac{\pi }{2n+1}\sin \frac{2\pi }{2n+1}$\\
    $0$&$-\sqrt{\sigma\tau }$&$4\sqrt{|\sigma\tau|}\sin \frac{\pi }{2n+1}\sin \frac{2\pi }{2n+1}$\\
    $-\sqrt{\sigma\tau }$&$0$&$4\sqrt{|\sigma\tau|}\sin \frac{\pi }{2n+1}\sin \frac{2\pi }{2n+1}$\\
    $\sqrt{\sigma\tau }$&$-\sqrt{\sigma\tau }$&$4\sqrt{|\sigma\tau|}\sin ^2\frac{\pi }{n}$\\
    $-\sqrt{\sigma\tau }$&$\sqrt{\sigma\tau }$&$4\sqrt{|\sigma\tau|}\sin^2 \frac{\pi }{n}$\\
    $\sqrt{\sigma\tau }$&$\sqrt{\sigma\tau }$&$4\sqrt{|\sigma\tau|}\sin^2 \frac{\pi }{2n}$\\
    $-\sqrt{\sigma\tau }$&$-\sqrt{\sigma\tau }$&$4\sqrt{|\sigma\tau|}\sin^2 \frac{\pi }{2n}$
\end{tabular}
\end{center}
\caption{Minimal distance between eigenvalues of the matrix (\ref{Talbeta}) for several 
choices of $\alpha$ and $\beta$.}\label{table2}
\end{table}

\subsection{Eigenvalue condition numbers}
We first consider the eigenvalues of the Toeplitz matrix (\ref{TRap}). Condition numbers
for these eigenvalues also have been discussed in \cite{NPR}. When $\sigma\tau\neq 0$, 
eigenvalue condition numbers can be obtained from (\ref{xhk}) and (\ref{yhk}). Standard 
computations and the trigonometric identity 
\begin{equation}\label{trig1}
\sum_{k=1}^{n}\sin^{2}\frac{hk\pi }{n+1}=
\frac{n+1}{2} \,,\quad h=1:n, 
\end{equation}
show that, for $h=1:n$,
\begin{eqnarray*}
\|x_{h}\|_{2}^{2}&=&\sum_{k=1}^{n}\left|\frac{\sigma}{\tau}\right|^{k}
\sin^{2} \frac{hk\pi}{n+1}, \\
\|y_{h}\|_{2}^{2}&=&\sum_{k=1}^{n}\left|\frac{\tau}{\sigma}\right|^{k}
\sin ^{2}\frac{hk\pi}{n+1},\\
|y_{h}^{H}\,x_{h}|&=&\sum_{k=1}^{n}\sin^{2}\frac{hk\pi}{n+1}
=\frac{n+1}{2}.
\end{eqnarray*}
Consequently, the condition numbers for the eigenvalues $\lambda_h$, $h=1:n$,
of the matrix (\ref{TRap}) are given by
\begin{eqnarray}
\kappa(\lambda_{h} )&=&
\frac{\|x_{h}\|_{2}\|y_{h}\|_{2}}{\left|y_{h}^{H}\,x_{h}\right|}\\
\label{klamhT}
&=&
\nonumber
\frac{2}{n+1}\sqrt{\sum_{k=1}^{n}
\left|\frac{\sigma}{\tau}\right|^{k}\sin^{2} \frac{hk\pi}{n+1}
\cdot \sum_{k=1}^{n}\left|\frac{\tau}{\sigma}\right|^{k}\sin^{2}
\frac{hk\pi}{n+1}}.
\end{eqnarray}
Note that the eigenvalue condition numbers $\kappa(\lambda_{h})$ only depend on $h$, $n$,
and the ratio $\left|\frac{\sigma}{\tau}\right|$. When $|\sigma|=|\tau|$, we have
\[
\left\| x_{h}\right\| _{2}^{2}=\left\| y_{h}\right\|_{2}^{2}=
\sum_{k=1}^{n}\sin ^{2} \frac{hk\,\pi }{n+1}=
\frac{n+1}{2}\,,\quad h=1:n,
\]
and it follows that 
\[
\kappa(\lambda_{h} )=\frac{\left\| x_{h}\right\|_{2}
\left\|y_{h}\right\|_{2}}{\left|y_{h}^{H}\,x_{h}\right|}=1.
\]
Thus, the eigenvalues are perfectly conditioned. This is in agreement with the observation 
that the matrix $T$ is normal when $|\sigma|=|\tau|$; see \cite[Theorem 3.1]{NPR}.

We turn to the condition numbers of the eigenvalues of the Toeplitz-like matrix 
$T_{\alpha,\beta}$ defined by (\ref{Talbeta}) for parameters $\alpha$ and $\beta$ 
considered in Table \ref{table1}. The condition numbers, which are reported in 
Tables \ref{table3} and \ref{table4}, can be derived by using the trigonometric 
identities
\[
\begin{array}{rclcrcl}
\sum_{k=1}^{n}\sin^{2} \frac{2hk\,\pi }{2n+1}&=&\frac{2n+1}{4}, &\quad& 
\sum_{k=1}^{n}\sin^{2} \frac{h(2k-1)\,\pi}{2n+1}&=&\frac{2n+1}{4},\\
\sum_{k=1}^{n}\sin^{2} \frac{(2h-1)k\,\pi}{2n+1}&=&\frac{2n+1}{4}, &\quad&
\sum_{k=1}^{n}\cos^{2} \frac{(2h-1)(2k-1)\,\pi}{2(2n+1)}&=& \frac{2n+1}{4},\\
\sum_{k=1}^{n}\sin^{2} \frac{(2h-1)(2k-1)\,\pi}{4n}&=&\frac{n}{2},&\quad&
\sum_{k=1}^{n}\cos^{2} \frac{(2h-1)(2k-1)\,\pi}{4n}&=& \frac{n}{2},
\end{array}
\]
for $h=1:n$, and 
\begin{equation*}
\sum_{k=1}^{n}\sin ^{2} \frac{h(2k-1)\,\pi }{2n}= \frac{n}{2}\,, \textrm{ if  } h\neq n;
\quad \sum_{k=1}^{n}\cos ^{2} \frac{(h-1)(2k-1)\,\pi }{2n}= \frac{n}{2},  \textrm{ if  } 
h\neq 1.
\end{equation*}
Notice that $\kappa(\lambda_h)=1$ when $\alpha=\beta=\sqrt{\sigma\tau}$ and $h=n$, and 
when $\alpha=\beta=-\sqrt{\sigma\tau}$ and $h=1$. 

\begin{proposition}\label{propnew}
Let $|\sigma|=|\tau|>0$, and let $\alpha$ and $\beta$ be defined as in Table \ref{table1}.
Then the tridiagonal Toeplitz-like matrix given by \eqref{Talbeta} is normal. 
Consequently, all eigenvalues have condition number one.
\end{proposition}

\begin{proof}
To show normality, we may apply \cite[Theorem 1]{BF} and note, using the notation of this
reference, that we have $\delta\pm\sqrt{\sigma\tau}=(\hat{r}+id)e^{i\phi}$, with 
$\hat{r}=r\pm|\sigma|$, if $\delta= (r+id)e^{i\phi}$ and $\tau=\sigma e^{2i\phi}$ for 
some $r,d,\phi\in\mathbb{R}$, and if $|\sigma|=|\tau|\neq 0$. Here and below $i$ denotes 
the imaginary unit.
\end{proof}

\begin{table}
{\small
\begin{center}
    \begin{tabular}{c|c|c}
    $\alpha$ & $\beta$& $\kappa(\lambda_h)$ \\
    \hline 
    $0$&$\sqrt{\sigma\tau }$&$\frac{4}{2n+1}\sqrt{\sum_{k=1}^{n}
\left|\frac{\sigma}{\tau}\right|^{k}\sin^{2} \frac{2hk\pi}{2n+1}
\cdot \sum_{k=1}^{n}\left|\frac{\tau}{\sigma}\right|^{k}\sin^{2}
\frac{2hk\pi}{2n+1}}$\\
    $\sqrt{\sigma\tau }$&$0$&$\frac{4}{2n+1}\sqrt{\sum_{k=1}^{n}
\left|\frac{\sigma}{\tau}\right|^{k}\sin^{2} \frac{h(2k-1)\pi}{2n+1}
\cdot \sum_{k=1}^{n}\left|\frac{\tau}{\sigma}\right|^{k}\sin^{2}
\frac{h(2k-1)\pi}{2n+1}}$\\
    $0$&$-\sqrt{\sigma\tau }$&$\frac{4}{2n+1}\sqrt{\sum_{k=1}^{n}
\left|\frac{\sigma}{\tau}\right|^{k}\sin^{2} \frac{(2h-1)k\pi}{2n+1}
\cdot \sum_{k=1}^{n}\left|\frac{\tau}{\sigma}\right|^{k}\sin^{2}
\frac{(2h-1)k\pi}{2n+1}}$\\
    $-\sqrt{\sigma\tau }$&$0$&$\frac{4}{2n+1}\sqrt{\sum_{k=1}^{n}
\left|\frac{\sigma}{\tau}\right|^{k}\cos^{2} \frac{(2h-1)(2k-1)\pi}{2(2n+1)}
\cdot \sum_{k=1}^{n}\left|\frac{\tau}{\sigma}\right|^{k}\cos^{2}
\frac{(2h-1)(2k-1)\pi}{2(2n+1)}}$\\
    $\sqrt{\sigma\tau }$&$-\sqrt{\sigma\tau }$&$\frac{2}{n}\sqrt{\sum_{k=1}^{n}
\left|\frac{\sigma}{\tau}\right|^{k}\sin^{2} \frac{(2h-1)(2k-1)\pi}{4n}
\cdot \sum_{k=1}^{n}\left|\frac{\tau}{\sigma}\right|^{k}\sin^{2}
\frac{(2h-1)(2k-1)\pi}{4n}}$\\
    $-\sqrt{\sigma\tau }$&$\sqrt{\sigma\tau }$&$\frac{2}{n}\sqrt{\sum_{k=1}^{n}
\left|\frac{\sigma}{\tau}\right|^{k}\cos^{2} \frac{(2h-1)(2k-1)\pi}{4n}
\cdot \sum_{k=1}^{n}\left|\frac{\tau}{\sigma}\right|^{k}\cos^{2}
\frac{(2h-1)(2k-1)\pi}{4n}}$
        \end{tabular}
\end{center}}
\caption{Condition numbers of the eigenvalues $\lambda_h$, $h=1:n$, of the matrix
(\ref{Talbeta}) for several choices of $\alpha\neq\beta$.}
\label{table3}
\end{table}

\begin{table}
\begin{center}
    \begin{tabular}{c|c|c}
    $\alpha$ & $\beta$& $\kappa(\lambda_h)$  \\
    \hline 
    $\sqrt{\sigma\tau }$&$\sqrt{\sigma\tau }$&$\frac{2}{n}\sqrt{\sum_{k=1}^{n}
\left|\frac{\sigma}{\tau}\right|^{k}\sin^{2} \frac{h(2k-1)\pi}{2n}
\cdot \sum_{k=1}^{n}\left|\frac{\tau}{\sigma}\right|^{k}\sin^{2}
\frac{h(2k-1)\pi}{2n}}$\\
    $-\sqrt{\sigma\tau }$&$-\sqrt{\sigma\tau }$&$\frac{2}{n}\sqrt{\sum_{k=1}^{n}
\left|\frac{\sigma}{\tau}\right|^{k}\cos^{2} \frac{(h-1)(2k-1)\pi}{2n}
\cdot \sum_{k=1}^{n}\left|\frac{\tau}{\sigma}\right|^{k}\cos^{2}
\frac{(h-1)(2k-1)\pi}{2n}}$\\
        \end{tabular}
\end{center}
\caption{Condition numbers of the eigenvalues $\lambda_h$, $h=1:n-1$, of the matrix
(\ref{Talbeta}) for $\alpha=\beta=\sqrt{\sigma\tau}$, and of the eigenvalues 
$\lambda_h$, $h=2:n$, for $\alpha=\beta=-\sqrt{\sigma\tau }$.}\label{table4}
\end{table}

\section{Sensitivity of the eigenvectors}\label{sec3}
The beginning of this section reviews results by Stewart \cite{St01}. These results are 
subsequently applied to the matrices \eqref{TRap} and \eqref{Talbeta}.

\subsection{Eigenvector condition numbers}\hfill\break
\begin{definition}(\cite{St01})\label{condvec1}
Let $A\in\mathbb{C}^{n \times n}$ and let $x$ be an eigenvector of unit norm associated 
with the simple eigenvalue $\mu$. Let $U\in\mathbb{C}^{n \times (n-1)}$ be a matrix whose 
columns form an orthonormal basis for $\text{Range}(A-\mu I)$. The condition number of $x$
(i.e., the condition number of the one-dimensional invariant subspace spanned by $x$) is
defined by
\[
\kappa(x)=\|(\mu I - U^HAU)^{-1}\|_2^{-1}.
\]
\end{definition}

Let $A^{\varepsilon}=A+\varepsilon E$, where $\varepsilon\in\R$ is of small magnitude and
$E\in{\mathbb C}^{n\times n}$ is a matrix with $\|E\|_F=1$. Here and below $\|\cdot\|_F$ 
denotes the Frobenius norm. Let $x^{\varepsilon}$ be the unit eigenvector of 
$A^{\varepsilon}$ corresponding to $x$, i.e., there is a continuous mapping $t\to x^t$ for 
$0\leq t\leq\varepsilon$ such that $x^t=x$ for $t=0$ and $x^t=x^\varepsilon$ for 
$t=\varepsilon$. Then for the induced perturbation in the direction between $x$ and the 
pseudoeigenvector $x^{\varepsilon}$ one has
\begin{equation}
\sin \theta_{x,x^{\varepsilon}}\leq \kappa(x)\varepsilon,
\end{equation}
where $\sin \theta_{x,x^{\varepsilon}}:= \sqrt{1-\cos^2 \theta_{x,x^{\varepsilon}}}$ and 
$\cos \theta_{x,x^{\varepsilon}}:=|x^H x^{\varepsilon}|$; see Stewart 
\cite[pp. 48--50]{St01} for more details.

\subsection{Eigenvector condition numbers in the normal case}
Let the matrix $A\in{\mathbb C}^{n\times n}$ be normal and denote its spectrum by 
$\Lambda(A)$. Consider the expression $\|(\mu I-U^HAU)^{-1}\|_2$ of 
Definition \ref{condvec1}. It is straightforward to show that the upper bound
\[
\|(\mu I - U^HAU)^{-1}\|_2 \leq \min_{{\substack{\lambda\neq\mu\\ \lambda\in\Lambda(A)}}}
|\mu-\lambda|
\]
is attained because $A$ is unitarily diagonalizable. Thus, if the matrix 
$A\in{\mathbb C}^{n\times n}$ is normal, the condition number of a unit eigenvector $x$ only
depends on how well the associated eigenvalue $\mu$ is separated from the 
other eigenvalues of the matrix. This result leads to the following proposition, which
is shown in \cite{St01}.

\begin{proposition}\label{codvecdef}
Let $A\in\mathbb{C}^{n \times n}$ be a normal matrix and let $x$ be a unit eigenvector 
associated with the simple eigenvalue $\mu$.  The condition number of $x$ (i.e., the 
condition number of the one-dimensional invariant subspace spanned by $x$) is given by
\[
\kappa(x)=\left( \min_{\substack{\lambda\neq\mu\\ \lambda\in\Lambda(A)}} 
|\mu-\lambda|\right)^{-1}.
\]
\end{proposition}

Consider the tridiagonal Toeplitz matrix $T=(n;\sigma,\delta,\tau)$ and introduce the 
right and left unit eigenvectors,
\[
\widetilde{x}_h=\frac{x_h}{\|x_h\|},\quad \widetilde{y}_h=\frac{y_h}{\|y_h\|},
\quad h=1:n,
\]
where the vectors $x_h$ and $y_h$ are defined by (\ref{xhk}) and (\ref{yhk}), 
respectively. 

\begin{proposition}
Let the Toeplitz matrix $T=(n;\sigma,\delta,\tau)$ be normal. Then the condition number 
of $\widetilde{x}_h$ is given by
\begin{equation}\label{condvec}
\kappa(\widetilde{x}_h)=
\begin{cases}
        \left(4|\sigma|\sin \frac{\pi }{2(n+1)}\sin \frac{(2h-1)\pi }{2(n+1)}\right)^{-1}, & \text{for }   1< h\leq \frac n 2 \text{  or } h=n,\\
        \\
        \left(4|\sigma|\sin \frac{\pi }{2(n+1)}\sin \frac{(2h+1)\pi }{2(n+1)}\right)^{-1}, & \text{for }  h=1 \text{  or } \frac n 2< h<n.\\
\end{cases}
\end{equation} 
In particular, $\kappa(\widetilde{x}_h)$ depends only on $h$, $n$, and $|\sigma|$.
Moreover,
\[
\max_{h=1:n}\kappa(\widetilde{x}_h)=\left(4|\sigma|\sin\frac{\pi }{2(n+1)}
\sin\frac{3\pi }{2(n+1)}\right)^{-1}.
\]
The maximum is attained by the eigenvectors $\widetilde{x}_h$ associated with the four 
extremal eigenvalues with indices $h=1,2,n-1,n$. 
\end{proposition}

\begin{proof} 
The proof follows from Propositions \ref{prop1} and \ref{codvecdef}, by using the 
characterization in \cite[Theorem 3.1]{NPR}.
\end{proof}

Figure \ref{fig2} shows the condition numbers $\kappa(\widetilde{x}_h)$ of normalized eigenvectors
of a $100\times 100$ normal tridiagonal Toeplitz matrix with $|\sigma|=|\tau|=1$.

\begin{figure}[tbp]
\centering
\includegraphics[scale=0.60]{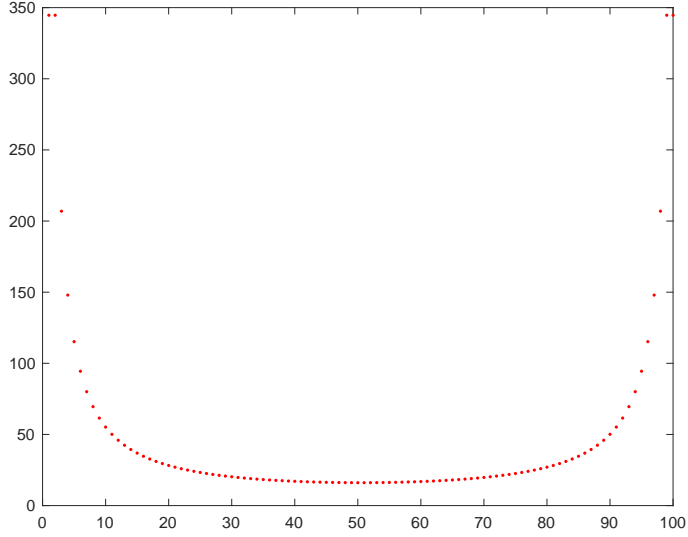}
\caption{Eigenvector condition numbers  for the matrix $T=(100;\exp i\theta_1,\delta,\exp i\theta_2)$, 
where $\delta\in\C$ and $\theta_1, \theta_2\in \R$ are arbitrarily chosen parameters, and $i=\sqrt{-1}$. 
The horizontal axis shows the index of the eigenvalues $\lambda_h$, $h=1:100$, and the vertical axis 
shows the condition numbers $\kappa(\widetilde{x}_h)$. The condition numbers are independent of 
$\delta$, $\theta_1$, and $\theta_2$.}\label{fig2}
\end{figure}

Let $T^{\varepsilon}=T+\varepsilon E$, where $\varepsilon\in\R$ is a constant of small 
magnitude, and $E\in{\mathbb C}^{n\times n}$ satisfies $\|E\|_F=1$. Introduce the 
unit pseudoeigenvector $\widetilde{x}^{\varepsilon}_h$ of $T^{\varepsilon}$ 
corresponding to the unit eigenvector $\widetilde{x}_h$ of $T$. Thus, there is a 
continuous mapping $t\to\widetilde{x}_h^t$ for $0\leq t\leq\varepsilon$ such that 
$\widetilde{x}_h^t=\widetilde{x}_h$ for $t=0$ and 
$\widetilde{x}_h^t=\widetilde{x}_h^\varepsilon$ for $t=\varepsilon$. We obtain from
(\ref{condvec}) that
\begin{equation}
0\leq\sin \theta_{\widetilde{x}_h,\widetilde{x}^{\varepsilon}_h}\leq
\begin{cases}
        \left(4|\sigma|\sin \frac{\pi }{2(n+1)}\sin \frac{(2h-1)\pi }{2(n+1)}\right)^{-1}\varepsilon, & \text{for }   1< h\leq \frac n 2 \text{  or } h=n,\\
        \\
        \left(4|\sigma|\sin \frac{\pi }{2(n+1)}\sin \frac{(2h+1)\pi }{2(n+1)}\right)^{-1}\varepsilon, & \text{for }  h=1 \text{  or } \frac n 2< h<n.\\
\end{cases}
\end{equation} 

\begin{proposition}\label{sinherm}
Let the matrix $T=(n;\sigma,\delta,\tau)$ be Hermitian. Given the unit pseudoeigenvector 
$\widetilde{x}^{\varepsilon}_h$, define the associated Rayleigh quotient,
\[
\widetilde{\lambda}_h^{\varepsilon}=
(\widetilde{x}^{\varepsilon}_h)^HT\widetilde{x}^{\varepsilon}_h,
\]
and introduce the associated residual norm
\[
r_h^{\varepsilon}=\|T \widetilde{x}^{\varepsilon}_h-\widetilde{\lambda}_h^{\varepsilon} 
\widetilde{x}^{\varepsilon}_h\|_2.
\]
Then
\begin{equation}
\frac{r_h^{\varepsilon}}{2|\sigma|\cos \frac{\pi}{n+1}} \leq 
\sin \theta_{\widetilde{x}_h,\widetilde{x}^{\varepsilon}_h}\leq 
\frac{r_h^{\varepsilon}}{\min_{k\neq h}|\lambda_k-\widetilde{\lambda}_h^{\varepsilon}|}\,.
\end{equation}
\end{proposition}

\begin{proof}
The proof follows from \cite[Theorem 11.7.1]{P98} by observing that 
$\text{spread}(T):=\lambda_1-\lambda_n= 2|\sigma|\cos\frac{\pi}{n+1}$.
\end{proof}

We turn to the condition number of the eigenvectors of the matrix (\ref{Talbeta}) for
$\alpha$- and $\beta$-values of Table \ref{table1}. Using Proposition \ref{prop1bis}, we 
obtain the following expressions.

\begin{proposition}
Let the matrix $T_{\alpha,\beta}\in{\mathbb C}^{n\times n}$ be normal and let
$\widetilde{x}_h$, for $h=1:n$, be unit eigenvectors. Then
\begin{enumerate}[label=(\roman*)]
\item for $\alpha=0$ and $ \beta=\sqrt{\sigma\tau }$ or vice versa

\begin{equation*}
\kappa(\widetilde{x}_h)=
\begin{cases}
\left(4\sqrt{|\sigma\tau|}\sin \frac{\pi }{2n+1}\sin \frac{(2h-1)\pi }{2n+1}\right)^{-1}, & \text{for }   1< h\leq \frac n 2 \text{  or } h=n,\\
        \\
\left( 4\sqrt{|\sigma\tau|}\sin \frac{\pi }{2n+1}\sin \frac{(2h+1)\pi }{2n+1}\right)^{-1}, & \text{for }  h=1 \text{  or } \frac n 2< h<n;\\        
\end{cases}
\end{equation*} 

\item for $\alpha=0$ and $ \beta=-\sqrt{\sigma\tau }$ or vice versa

\begin{equation*}
\kappa(\widetilde{x}_h)= 
\begin{cases}
        \left(4\sqrt{|\sigma\tau|}\sin \frac{\pi }{2n+1}\sin \frac{2(h-1)\pi }{2n+1}\right)^{-1}, & \text{for }   1< h\leq \lceil \frac n 2 \rceil \text{  or } h=n,\\
        \\
        \left(4\sqrt{|\sigma\tau|}\sin \frac{\pi }{2n+1}\sin \frac{2h\pi }{2n+1}\right)^{-1}, & \text{for }  h=1 \text{  or } \lceil \frac n 2 \rceil< h<n;\\
\end{cases}
\end{equation*}

\item for $\alpha=\sqrt{\sigma\tau }$ and $ \beta=-\sqrt{\sigma\tau }$ or vice versa

\begin{equation*}
\kappa(\widetilde{x}_h)= 
\begin{cases}
        \left(4\sqrt{|\sigma\tau|}\sin \frac{\pi }{n}\sin \frac{(h-1)\pi }{n}\right)^{-1}, & \text{for }   1< h\leq \frac n 2 \text{  or } h=n,\\
        \\
        \left(4\sqrt{|\sigma\tau|}\sin \frac{\pi }{n}\sin \frac{h\pi }{n}\right)^{-1}, & \text{for }  h=1 \text{  or } \frac n 2< h<n;\\
\end{cases}
\end{equation*}

\item for $\alpha=\sqrt{\sigma\tau }$ and $ \beta=\sqrt{\sigma\tau }$

\begin{equation*}
\kappa(\widetilde{x}_h)= 
\begin{cases}
        \left(4\sqrt{|\sigma\tau|}\sin \frac{\pi }{2n}\sin \frac{(2h-1)\pi }{2n}\right)^{-1}, & \text{for }   1< h\leq \frac n 2 \text{  or } h=n,\\
        \\
        \left(4\sqrt{|\sigma\tau|}\sin \frac{\pi }{2n}\sin \frac{(2h+1)\pi }{2n}\right)^{-1}, & \text{for }  h=1 \text{  or } \frac n 2< h<n;\\
\end{cases}
\end{equation*}
\item for $\alpha=-\sqrt{\sigma\tau }$ and $ \beta=-\sqrt{\sigma\tau }$

\begin{equation*}
\kappa(\widetilde{x}_h)=  
\begin{cases}
        \left(4\sqrt{|\sigma\tau|}\sin \frac{\pi }{2n}\sin \frac{(2h-3)\pi }{2n}\right)^{-1}, & \text{for }   1< h\leq \lceil \frac n 2 \rceil \text{  or } h=n,\\
        \\
        \left(4\sqrt{|\sigma\tau|}\sin \frac{\pi }{2n}\sin \frac{(2h-1)\pi }{2n}\right)^{-1}, & \text{for }  h=1 \text{  or } \lceil \frac n 2 \rceil< h<n;\\
\end{cases}
\end{equation*}
\end{enumerate}
The maximal eigenvector condition numbers are reported in Table \ref{table5}.
\end{proposition}

\begin{table}
\begin{center}
    \begin{tabular}{c|c|c}
    $\alpha$ & $\beta$& $\max_{h=1:n}\kappa(\widetilde{x}_h)$  \\
    \hline 
    $0$&$\sqrt{\sigma\tau }$&$ \left(4|\sigma|\sin \frac{\pi }{2n+1}\sin \frac{2\pi }{2n+1}\right)^{-1}$\\
    $\sqrt{\sigma\tau }$&$0$&$ \left(4|\sigma|\sin \frac{\pi }{2n+1}\sin \frac{2\pi }{2n+1}\right)^{-1}$\\
    $0$&$-\sqrt{\sigma\tau }$&$ \left(4|\sigma|\sin \frac{\pi }{2n+1}\sin \frac{2\pi }{2n+1}\right)^{-1}$\\
    $-\sqrt{\sigma\tau }$&$0$&$ \left(4|\sigma|\sin \frac{\pi }{2n+1}\sin \frac{2\pi }{2n+1}\right)^{-1}$\\
    $\sqrt{\sigma\tau }$&$-\sqrt{\sigma\tau }$&$ \left(4|\sigma|\sin ^2\frac{\pi }{n}\right)^{-1}$\\
    $-\sqrt{\sigma\tau }$&$\sqrt{\sigma\tau }$&$ \left(4|\sigma|\sin^2 \frac{\pi }{n}\right)^{-1}$\\
    $\sqrt{\sigma\tau }$&$\sqrt{\sigma\tau }$&$ \left(4|\sigma|\sin^2 \frac{\pi }{2n}\right)^{-1}$\\
    $-\sqrt{\sigma\tau }$&$-\sqrt{\sigma\tau }$&$ \left(4|\sigma|\sin^2 \frac{\pi }{2n}\right)^{-1}$
\end{tabular}
\end{center}
\caption{Maximal eigenvector condition numbers for the eigenvectors of the matrix 
(\ref{Talbeta}) for several choices of $\alpha$ and $\beta$.}\label{table5}
\end{table}

\section{Sensitivity to structured perturbations} \label{sec4}
Observe that the smaller $0<|\sigma/\tau|<1$ is, the larger is the first component of the
unit right eigenvector $\widetilde{x}_h$ and the last component of the unit left 
eigenvector $\widetilde{y}_h$. Similarly, the larger $1<|\sigma/\tau|<\infty$ is, the 
larger is the last component of $\widetilde{x}_h$ and the first component of 
$\widetilde{y}_h$. 

Consider the Wilkinson perturbation 
\[
W_h=\widetilde{y}_h\widetilde{x}_h^H
\]
of the matrix $T$ defined by (\ref{TRap}) associated with the eigenvalue $\lambda_h$. This
is a unit-norm perturbation of $T$ that yields the largest perturbation in $\lambda_h$; 
see, e.g., \cite{Wi}. The entries of largest magnitude of $W_h$ are in the bottom-left
corner of $W_h$ when $|\sigma/\tau|<1$ and in the top-right corner when $|\sigma/\tau|>1$.
The entries of $W_h$ close to the diagonal are of small magnitude. In particular, the 
entries of largest magnitude of $W_h$ are not present in $W_h |_{{\mathcal T}}$, the 
orthogonal projection of $W_h$ in the subspace ${\mathcal T}$ of tridiagonal Toeplitz 
matrices. This projection is used in the following proposition which summarizes results 
from \cite{NP2} and yields useful formulations for both the ${\mathcal T}$-structured 
eigenvalue condition number (see, e.g., \cite{KKT}) and the worst-case ${\mathcal T}$-structured 
perturbations \cite{NP1,NP2}.

\begin{proposition}\label{lem:p1}
Let $\lambda_h$ be a simple eigenvalue of a Toeplitz matrix 
$T\in{\mathcal T}\subset\mathbb{C}^{n\times n}$ with associated unit right and left eigenvectors 
$\widetilde{x}_h$ and $\widetilde{y}_h$, respectively. Given any matrix 
$E\in{\mathcal T}$ with $\|E\|_F = 1$, let $\lambda_h(t)$ be an eigenvalue of $T+tE$ converging
to $\lambda_h$ as $t\to 0$. Then
\[
|\dot\lambda_h(0)| \le \max\left\{\left|\frac{\widetilde{y}_h^HE\widetilde{x}_h}{\widetilde{y}_h^H\widetilde{x}_h}\right| , \; \|E\|_F = 1,\, 
E \in \mathcal T\,\right\} = \frac{\|W_h |_{{\mathcal T}}\|_F}{|\widetilde{y}_h^H\widetilde{x}_h|}
\]
and
\[
\dot\lambda_h(0)=\frac{\|W_h |_{{\mathcal T}}\|_F}{|\widetilde{y}_h^H\widetilde{x}_h|} 
\qquad \mathrm{if} \qquad E=\eta\frac{W_h |_{{\mathcal T}}}{\|W_h |_{{\mathcal T}}\|_F},
\]
for any unimodular $\eta \in {\mathbb C}$. Here $\dot\lambda_h(t)$ denotes the derivative 
of $\lambda_h(t)$ with respect to the parameter $t$.
\end{proposition}

It follows from Proposition \ref{lem:p1} that the ${\mathcal T}$-structured condition number of the eigenvalue 
$\lambda_h$ of the tridiagonal Toeplitz matrix $T$ is given by
\[
\kappa_{{\mathcal T}}(\lambda_{h} )=\kappa(\lambda_{h})\|W_h |_{{\mathcal T}}\|_F.
\]
This expression shows that the ${\mathcal T}$-structured condition number 
$\kappa_{{\mathcal T}}(\lambda_{h})$ may be small even when the traditional condition
number $\kappa (\lambda_h)$ is large. Thus, an eigenvalue $\lambda_h $ may be much more 
sensitive to a general perturbation of $T$ than to a structured perturbation.  
The worst-case structured perturbation \cite{NP2} is given by the structured analogue of
the Wilkinson perturbation  
\[
W_h |_{\widehat{\mathcal T}}:=\frac{W_h |_{{\mathcal T}}}{\|W_h |_{{\mathcal T}}\|_F}.
\]
We have the following result.

\begin{proposition}\label{propstr}
The ${\mathcal T}$-structured condition number of a simple  eigenvalue $\lambda_h$ of a
tridiagonal Toeplitz matrix $T=(n;\sigma,\delta,\tau)$ is given by
\begin{equation}\label{condnorm}
\kappa_{\mathcal T}(\lambda_{h} )=\sqrt{\frac{1}{n}+\frac{1}{n-1}\left(\left|\frac{\sigma}{\tau}\right|+\left|\frac{\tau}{\sigma}\right|\right)\cos^2 \frac{h\pi}{n+1}}.
\end{equation} 
In particular, $\kappa_{\mathcal T}(\lambda_{h} )$ only depends on $h$, $n$, and the 
ratio $\left|\frac{\sigma}{\tau}\right|$.
\end{proposition}

\begin{proof}
Let $\sigma_h$, $\delta_h$, and $\tau_h$ denote the subdiagonal, diagonal, and 
superdiagonal entries of $W_h|_{{\mathcal T}}$, respectively. Then
\begin{eqnarray*}
\sigma_h&=&\frac{\sqrt{\frac{\bar{\tau}}{\bar{\sigma}}}
\sum_{k=1}^{n-1}\sin\frac{hk\pi}{n+1}\sin\frac{h(k+1)\pi}{n+1}}{(n-1)
\sqrt{\sum_{k=1}^{n} \left|\frac{\sigma}{\tau}\right|^{k}\sin^{2} \frac{hk\pi}{n+1}
\cdot \sum_{k=1}^{n}\left|\frac{\tau}{\sigma}\right|^{k}\sin^{2}
\frac{hk\pi}{n+1}}}\\
&=&\frac{(n+1)\sqrt{\frac{\bar{\tau}}{\bar{\sigma}}}
\cos \frac{h\pi}{n+1} }{2(n-1)\sqrt{\sum_{k=1}^{n}
\left|\frac{\sigma}{\tau}\right|^{k}\sin^{2} \frac{hk\pi}{n+1}
\cdot \sum_{k=1}^{n}\left|\frac{\tau}{\sigma}\right|^{k}\sin^{2}
\frac{hk\pi}{n+1}}}, \\
\delta_h&=&\frac{n+1}{2n\sqrt{\sum_{k=1}^{n}
\left|\frac{\sigma}{\tau}\right|^{k}\sin^{2} \frac{hk\pi}{n+1}
\cdot \sum_{k=1}^{n}\left|\frac{\tau}{\sigma}\right|^{k}\sin^{2}
\frac{hk\pi}{n+1}}},\\ 
\tau_h&=&\frac{\sqrt{\frac{\bar{\sigma}}{\bar{\tau}}}\sum_{k=1}^{n-1} \sin \frac{hk\pi}{n+1} \sin \frac{h(k+1)\pi}{n+1}}{(n-1)\sqrt{\sum_{k=1}^{n}
\left|\frac{\sigma}{\tau}\right|^{k}\sin^{2} \frac{hk\pi}{n+1}
\cdot \sum_{k=1}^{n}\left|\frac{\tau}{\sigma}\right|^{k}\sin^{2}
\frac{hk\pi}{n+1}}}\\
&=&\frac{(n+1)\sqrt{\frac{\bar{\sigma}}{\bar{\tau}}}\cos \frac{h\pi}{n+1} }{2(n-1)\sqrt{\sum_{k=1}^{n}
\left|\frac{\sigma}{\tau}\right|^{k}\sin^{2} \frac{hk\pi}{n+1}
\cdot \sum_{k=1}^{n}\left|\frac{\tau}{\sigma}\right|^{k}\sin^{2}
\frac{hk\pi}{n+1}}}.
\end{eqnarray*}
The above expressions were obtained by exploiting the trigonometric identities 
\eqref{trig1} and
\begin{equation}\label{id}
\sum_{k=1}^{n-1} \sin \frac{hk\pi}{n+1} \sin \frac{h(k+1)\pi}{n+1}=
\frac{n+1}{2}\cos\frac{h\pi}{n+1},\qquad h=1:n;
\end{equation}
see, e.g., \cite[Appendix A]{BGN}. Hence,
\begin{eqnarray*}
\|W_h |_{{\mathcal T}}\|_F&=&\sqrt{n|\delta_h|^2+(n-1)|\sigma_h|^2+(n-1)|\tau_h|^2}\\\\
&=&\frac{\frac{n+1}{2}\sqrt{\frac{1}{n}+\frac{1}{n-1}
\left(\left|\frac{\sigma}{\tau}\right|+\left|\frac{\tau}{\sigma}\right|\right)
\cos^2 \frac{h\pi}{n+1}}}{\sqrt{\sum_{k=1}^{n}
\left|\frac{\sigma}{\tau}\right|^{k}\sin^{2} \frac{hk\pi}{n+1}
\cdot \sum_{k=1}^{n}\left|\frac{\tau}{\sigma}\right|^{k}\sin^{2}
\frac{hk\pi}{n+1}}}.
\end{eqnarray*}
Finally $\kappa_{{\mathcal T}}(\lambda_{h})$ is the product of $\kappa(\lambda_{h})$
and $\|W_h |_{{\mathcal T}}\|_F$. The proof now follows by using (\ref{klamhT}).
\end{proof}

\subsection{Eigenvector structured sensitivity in the normal case}
When $E$ is a (tridiagonal Toeplitz) structured perturbation of $T$, the perturbed matrix 
$T^{\varepsilon}=T+\varepsilon E$ is a tridiagonal Toeplitz matrix. Assume that $T$ is
normal. Unfortunately, $T^{\varepsilon}=(n;\sigma^\varepsilon,\delta^\varepsilon,\tau^\varepsilon)$ 
might not be normal because $|\sigma^{\varepsilon}|$ may differ from $|\tau^{\varepsilon}|$. For the 
components of the eigenvector $x^{\varepsilon}_h=[x^{\varepsilon}_{h,1},x^{\varepsilon}_{h,2},\ldots,
x^{\varepsilon}_{h,n}]^T$ associated with the $h$th eigenvalue of $T^{\varepsilon}$, we 
have 
\[
x^{\varepsilon}_{h,k}=\left(\sqrt{\frac{\sigma^{\varepsilon}}
{\tau^{\varepsilon}}}\right)^{k}\sin \frac{hk\pi}{n+1},\quad k=1:n,\quad h=1:n,
\]
so that
\begin{equation}\label{cosvec}
\cos\theta_{\widetilde{x}_h,\widetilde{x}^{\varepsilon}_h}=
\frac{\left|\sum_{k=1}^n \left(\sqrt{\frac{\bar{\sigma}}{\bar{\tau }}}\right)^{k}
\left(\sqrt{\frac{\sigma^{\varepsilon}}{\tau^{\varepsilon}}}\right)^{k}\sin^2 
\frac{hk\pi}{n+1}\right|}{\sqrt{\frac{n+1}{2}\sum_{k=1}^{n}
\left|\frac{\sigma^{\varepsilon}}{\tau ^{\varepsilon}}\right|^{k}\sin^{2}
\frac{hk\pi}{n+1}}} ,\quad
h=1:n,
\end{equation}
where $\widetilde{x}_h$ and $\widetilde{y}_h$ are normalized vectors.
Notice that the perturbations induced in the eigenvectors do not depend on 
$\delta^{\varepsilon}$. In fact, the induced perturbations only depend on the ratio 
$\frac{\sigma^{\varepsilon}}{\tau ^{\varepsilon}}$. 

\begin{proposition}
The right and left eigenvectors of normal tridiagonal Toeplitz matrices $T=(n; \sigma,\delta, \tau)$ only  depend on the dimension $n$ and on the angle $\theta=\arg(\sigma)-\arg(\tau)$.
\end{proposition}

\begin{proof}
From (\ref{xhk}) and (\ref{yhk}), it is clear that, given the dimension of the matrix, the
ratio $\sigma/\tau$ uniquely determines the right and left eigenvectors of $T$ up to a 
scaling factor. Since $|\sigma|=|\tau|$, one has 
\begin{equation}\label{xhksyme}
x_{h,k}=y_{h,k}=e^{i\frac{k}{2}\theta}\sin \frac{hk\pi}{n+1},\quad
k=1:n,\quad
h=1:n.
\end{equation}
\end{proof}

\begin{remark} \label{tht}
When $T$ is Hermitian, we have $\theta=2\arg(\sigma)$, whereas in the skew-Hermitian case, one has $\theta=2\arg(\sigma)-\pi$. 
\end{remark}

\begin{proposition}\label{herm}
If the perturbation $\varepsilon E$ of the Hermitian matrix 
$T=(n;\sigma,\delta,\bar{\sigma})$ has the same structure as $T$, then the right 
eigenvector $x^{\varepsilon}_h$ [the left eigenvector $y^{\varepsilon}_h$] associated to
the $h$th eigenvalue of 
\[
T^{\varepsilon}:=T+\varepsilon E=
(n;\sigma^{\varepsilon},\delta^{\varepsilon},\bar{\sigma}^{\varepsilon})
\]
has the components
\[
x^{\varepsilon}_{h,k}=y^{\varepsilon}_{h,k}=
e^{ik\text{arg}(\sigma^{\varepsilon})}\sin \frac{hk\pi}{n+1},\quad k=1:n,
\]
for $h=1:n$. Moreover, the associated Rayleigh quotient is given by 
\begin{equation}\label{lhkeps1}
\widetilde{\lambda}_h^{\varepsilon}:=\frac{x^{\varepsilon H}_hTx^{\varepsilon}_h}
{x^{\varepsilon H}_hx^{\varepsilon}_h}=\delta + 2|\sigma
|\cos(\text{arg}(\sigma)-\text{arg}(\sigma^{\varepsilon}))\cos \frac{h\pi}{n+1},
\quad h=1:n,
\end{equation}
and the following inequalities hold
$$
\frac{\|Tx^{\varepsilon}_h - \widetilde{\lambda}_h^{\varepsilon} 
x^{\varepsilon}_h\|_2}{\sqrt{2(n+1)}\cos \frac{\pi}{n+1}} \leq 
\sin \theta_{\widetilde{x}_h,\widetilde{x}^{\varepsilon}_h}\leq 
\frac{\|T x^{\varepsilon}_h - \widetilde{\lambda}_h^{\varepsilon} 
x^{\varepsilon}_h\|_2}{\sqrt{2(n+1)}|\sigma
|\left|1-\cos(\text{arg}(\sigma)-\text{arg}(\sigma^{\varepsilon}))\right|
\cos \frac{h\pi}{n+1}}\,.
$$
\end{proposition}
\begin{proof}
If $T$ is Hermitian, then $T^{\varepsilon}$ is Hermitian as well (i.e., 
$\tau^{\varepsilon}=\bar{\sigma^{\varepsilon}}$). The angle $\theta$ in (\ref{xhksyme}) is
equal to $2\arg(\sigma^{\varepsilon})$; see Remark \ref{tht}. Further, one has
$$
\widetilde{\lambda}_h^{\varepsilon}=\frac{\delta \sum_{k=1}^{n-1} \sin^2 \frac{hk\pi}{n+1} +(\sigma e^{-i\text{arg}(\sigma^{\varepsilon})}+\bar{\sigma}e^{i\text{arg}(\sigma^{\varepsilon})})\sum_{k=1}^{n-1} \sin \frac{hk\pi}{n+1} \sin \frac{h(k+1)\pi}{n+1}}{\frac{n+1}{2}}.
 $$
Exploiting the identities \eqref{trig1} and (\ref{id}), we obtain 
$$
\widetilde{\lambda}_h^{\varepsilon}=\delta+(\sigma e^{-i\text{arg}
(\sigma^{\varepsilon})}+\bar{\sigma}e^{i\text{arg}(\sigma^{\varepsilon})})
\cos\frac{h\pi}{n+1}.$$
Moreover (\ref{lhkeps1}) follows from $\sigma e^{-i\text{arg}(\sigma^{\varepsilon})}+
\bar{\sigma}e^{i\text{arg}(\sigma^{\varepsilon})}=|\sigma|
\Re(e^{i(\text{arg}(\sigma)-\text{arg}(\sigma^{\varepsilon}))})$, where $\Re(\cdot)$ denotes the real part of the argument.
The proof is concluded by using Proposition \ref{sinherm}, observing that 
$\Re(e^{it})=\cos(t)$,  $\|x^{\varepsilon}_h\|_2^2=\frac{n+1}{2}$, and
$$
\min_{k\neq h}|\lambda_k-\widetilde{\lambda}_h^{\varepsilon}|=
|\lambda_h-\widetilde{\lambda}_h^{\varepsilon}|=2|\sigma|\left|1-
\cos(\text{arg}(\sigma)-\text{arg}(\sigma^{\varepsilon}))\right|
\cos \frac{h\pi}{n+1}.
$$\end{proof}

When $T$ is  skew-Hermitian and the perturbation $E$ has the same structure,
we have that $T^{\varepsilon}$ is  skew-Hermitian as well (i.e., 
$\tau^{\varepsilon}=-\bar{\sigma}^{\varepsilon}$). Thus, in both the Hermitian and  
skew-Hermitian cases, the structured $\varepsilon$-pseudospectrum lies in a closed line 
segment, i.e., on the real axis or on the imaginary axis, respectively. In other situations
when $|\sigma|=|\tau|$ and $|{\sigma}^{\varepsilon}|\neq|\tau^{\varepsilon}|$, the structured $\varepsilon$-pseudospectrum is 
bounded by the ellipse $\{\tau z+\delta +\sigma z^{-1}: z\in{\mathbb C}, |z|=1\}$, which 
is the boundary of the spectrum of the Toeplitz operator 
$T_\infty=(\infty;\sigma,\delta,\tau )$; see, e.g., \cite{BGN,NPR,NR2,RT}.

\subsubsection{The real case}
The following results are concerned with with normal real  tridiagonal Toeplitz matrices.
\begin{proposition} \label{eqeig}
All real symmetric  tridiagonal Toeplitz matrices of a given dimension have the same right and left eigenvectors.
\end{proposition}

\begin{proof}
If $T$ is real and symmetric (i.e., $\sigma=\tau$), then 
\[
x_{h,k}=y_{h,k}=\sin \frac{hk\pi}{n+1},\quad k=1:n,\quad h=1:n.
\]
\end{proof}

\begin{corollary}\label{cors}
The eigenvectors of a real symmetric tridiagonal Toeplitz matrix are perfectly 
conditioned with respect to any structured perturbation that respects symmetry.
\end{corollary}

\begin{proof}
If $T$ is symmetric, then $T^{\varepsilon}$ is symmetric as well (i.e., 
$\sigma^{\varepsilon}=\tau^{\varepsilon}$). It follows from Proposition \ref{eqeig} that 
$\widetilde{x}_h=\widetilde{x}^{\varepsilon}_h$ for $h=1:n$. \end{proof}

\begin{corollary}\label{corss}
The eigenvectors of a real shifted skew-symmetric tridiagonal Toeplitz matrix are 
perfectly conditioned with respect to structured perturbations that respect both the 
skew-symmetry and the signs of the (sub- and) super-diagonals. 
\end{corollary}

\begin{proof}
If $T$ is shifted skew-symmetric, then  $\sigma=-\tau$, and one has 
\begin{equation*}
x_{h,k}=y_{h,k}=(\text{sgn}(\tau)i)^k\sin \frac{hk\pi}{n+1},\quad
k=1:n,\quad
h=1:n.
\end{equation*}
By assumption $T^{\varepsilon}$ is a real shifted skew-symmetric tridiagonal Toeplitz 
matrix and $\text{sgn}(\tau)=\text{sgn}(\tau^{\varepsilon})$. Thus, from (\ref{cosvec}), we
have 
$$\cos \theta_{\widetilde{x}_h,\widetilde{x}^{\varepsilon}_h}=
\frac{\left|\sum_{k=1}^n (\text{sgn}(\tau)i)^k(-\text{sgn}(\tau^{\varepsilon})i)^k\sin^2
\frac{hk\pi}{n+1}\right|}{\sqrt{\frac{n+1}{2} \sum_{k=1}^{n}\sin^{2}
\frac{hk\pi}{n+1}}}=\frac{\sum_{k=1}^n \sin^2 
\frac{hk\pi}{n+1}}{\sqrt{\frac{n+1}{2} \sum_{k=1}^{n}\sin^{2}\frac{hk\pi}{n+1}}}=1\,.
$$ \end{proof}

\begin{proposition}
The eigenvectors of a real normal tridiagonal Toeplitz matrix are perfectly conditioned 
with respect to any structured perturbation that respects the symmetry [skew-symmetry and 
signature]. 
\end{proposition}

\begin{proof}
A real tridiagonal matrix $T$ is normal if and only if it is symmetric or shifted 
skew-symmetric; see, e.g., \cite[Theorem 7.1] {NPR1} or \cite[Corollary 2.2]{NR}. The 
proof now follows from Corollaries \ref{cors} and \ref{corss}. 
\end{proof}

Let $\mathcal S_\mathcal T$ denote the subspace of real symmetric tridiagonal Toeplitz 
matrices and let $\mathcal A_\mathcal T$ be the subspace of real shifted skew-symmetric 
tridiagonal Toeplitz matrices. The above results show that the unstructured measure 
(\ref{condvec}) of the sensitivity to perturbations of the eigenvectors of a tridiagonal 
Toeplitz matrix in $\mathcal S_\mathcal T$ or $\mathcal A_\mathcal T$ is not accurate 
in case of structured perturbations $E$ of the matrix $T$, i.e., when 
$E \in \mathcal S_\mathcal T$ or $E \in \mathcal A_\mathcal T$ with $E$ small enough. 

\subsection{Eigenvalue structured sensitivity in the normal case}
 For normal matrices, the right and 
left unit eigenvectors can be chosen to be the same. Then the Wilkinson perturbation $W_h$ is symmetric for $h=1:n$.

\begin{corollary}
The ${\mathcal T}$-structured condition number of the eigenvalue $\lambda_h$ of a 
normal tridiagonal Toeplitz matrix $T$ is given by
\begin{equation}\label{condvecnorm}
\kappa_{\mathcal T}(\lambda_{h} )=\sqrt{\frac{1}{n}+\frac{2}{n-1} 
\cos^2 \frac{h\pi}{n+1}}, \quad h=1:n.
\end{equation} 
\end{corollary}

\begin{proof}
The proof trivially follows from (\ref{condnorm}), since $|\sigma|=|\tau|$.
\end{proof}

\subsubsection{The real case}
We recall that a real tridiagonal matrix $T$ is normal if and only if it is symmetric or shifted 
skew-symmetric.  Notice that Proposition  \ref{lem:p1} can be generalized to several other structures and that, 
in particular, it holds true if one everywhere replaces $\mathcal T$ by  either $\mathcal S_\mathcal T$ or  
$\mathcal A_\mathcal T$, or other subspaces of matrices with a given symmetry-pattern; see \cite{NP2}.  It 
follows that, for $h=1:n$, the ${\mathcal S_\mathcal T}$-structured [${\mathcal A_\mathcal T}$-structured] 
condition number of the eigenvalue $\lambda_h$ of a real symmetric [shifted 
skew-symmetric] tridiagonal Toeplitz matrix $T$ is given by
\[
\kappa_{{\mathcal S_\mathcal T}}(\lambda_{h} )=\|W_h |_{{\mathcal S_\mathcal T}}\|_F \qquad [\kappa_{{\mathcal A_\mathcal T}}(\lambda_{h} )=\|W_h |_{{\mathcal A_\mathcal T}}\|_F],
\]
$\kappa(\lambda_{h})$ being equal to $1$, and  that the worst-case structured perturbation \cite{NP2} is given by the structured analogue of
the Wilkinson perturbation:  
\begin{equation*}
W_h |_{\widehat{\mathcal S_\mathcal T}}:=\frac{W_h |_{{\mathcal S_\mathcal T}}}{\|W_h |_{{\mathcal S_\mathcal T}}\|_F}
\qquad [W_h |_{\widehat{\mathcal A_\mathcal T}}:=\frac{W_h |_{{\mathcal A_\mathcal T}}}{\|W_h |_{{\mathcal A_\mathcal T}}\|_F}].
\end{equation*}

The following result  is concerned with symmetric tridiagonal Toeplitz matrices and 
eigenvalue sensitivity to $\mathcal S_\mathcal T$-structured perturbations, i.e., to real symmetric 
tridiagonal Toeplitz  matrix perturbations.

\begin{proposition}\label{reals}
The eigenvalues $\lambda_h$ of any symmetric tridiagonal Toeplitz matrix 
$T\in\mathbb{R}^{n \times n}$  have  condition numbers 
$$\kappa_{{\mathcal S_\mathcal T}}(\lambda_h)=\sqrt{\frac{1}{n}+\frac{2}{n-1} 
\cos^2 \frac{h\pi}{n+1}}, \quad h=1:n,$$ 
with respect to any structured perturbation that respects the symmetry. 
\end{proposition}

\begin{proof}
It is straightforward that $\kappa_{{\mathcal S_\mathcal T}}(\lambda_{h})\leq \kappa_{\mathcal T}(\lambda_{h} )$. 
In addition, in the real symmetric case, i.e., when $\sigma = \tau$, the Wilkinson perturbation 
associated with $\lambda_h$, $W_h=\widetilde{y}_h\widetilde{x}_h^H$, is real and 
symmetric. Thus, the orthogonal projection of $W_h$ in the subspace  of real symmetric 
tridiagonal Toeplitz matrices coincides with $W_h |_{{\mathcal T}}$. This concludes the proof, since
$\kappa_{{\mathcal S_\mathcal T}}(\lambda_{h})$ 
coincides with the condition number $\kappa_{{\mathcal T}}(\lambda_{h} )$ in 
(\ref{condvecnorm}), i.e.,  
\[
\kappa_{{\mathcal S_\mathcal T}}(\lambda_{h})=\|W_h |_{{\mathcal S_\mathcal T}}\|_F=\|W_h |_{{\mathcal T}}\|_F= \kappa_{{\mathcal T}}(\lambda_{h} ).
\]
\end{proof}

Figure \ref{fig3} shows the structured eigenvalue condition numbers 
$\kappa_{{\mathcal S_\mathcal T}}(\lambda_{h})$ for a $100\times 100$ symmetric tridiagonal 
Toeplitz matrix. 
 
\begin{figure}[tbp]
\centering
\includegraphics[scale=0.60]{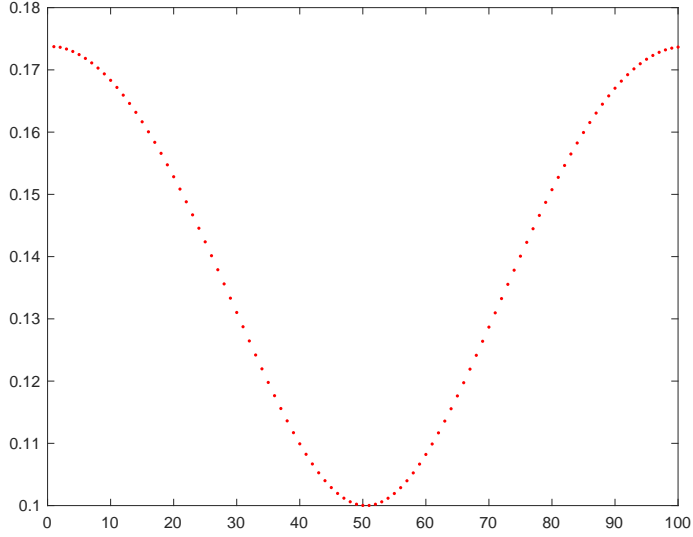}
\caption{Structured eigenvalue condition numbers for the matrix 
$T=(100;\sigma,\delta,\sigma)$, where  $\sigma$ and $\tau$ are arbitrarily chosen real parameters. 
The horizontal axis shows the index of the eigenvalue $\lambda_h$, $h=1:100$, and the vertical axis 
the structured condition numbers $\kappa_{{\mathcal S_\mathcal T}}(\lambda_{h})$. The condition 
numbers are independent of $\sigma$.}\label{fig3}
\end{figure}

\begin{remark}\label{rem_sym}
Let
$\sigma_h$, $\delta_h$, and $\tau_h$ denote the subdiagonal, diagonal, and superdiagonal 
entries, respectively,  of the orthogonal projection of the Wilkinson perturbation $W_h$  associated with the eigenvalue $\lambda_h$ of a real symmetric tridiagonal Toeplitz matrix  $T=(n; \sigma,\delta,\sigma)$ (i.e., $W_h |_{{\mathcal S_\mathcal T}}\equiv W_h |_{{\mathcal T}}$; cf. the proof of Proposition \ref{reals}). It is easy to show that
\begin{eqnarray*}
\sigma_h=\tau_h=\frac{1}{n-1} \cos \frac{h\pi}{n+1};\qquad \delta_h=\frac{1}{n}.
\end{eqnarray*}
Moreover, one has
\[
\widehat{\sigma}_h= \widehat{\tau}_h= \frac{\cos 
\frac{h\pi}{n+1}}{(n-1)\sqrt{\frac{1}{n}+\frac{2}{n-1} \cos^2 \frac{h\pi}{n+1}}};\quad
\widehat{\delta}_h= \frac{1}{n\sqrt{\frac{1}{n}+\frac{2}{n-1} \cos^2 
\frac{h\pi}{n+1}}},
\]
where $\widehat{\sigma}_h$, $\widehat{\tau}_h$, and $\widehat{\delta}_h$ denote 
the subdiagonal, diagonal, and superdiagonal entries, respectively, of the unit-norm
$\mathcal S_\mathcal T$-structured analogue of the Wilkinson perturbation, 
$W_h |_{\widehat{{\mathcal S_\mathcal T}}}$. Thus, if we perturb $T$ by the 
real symmetric tridiagonal Toeplitz matrix $\varepsilon W_{j} |_{\widehat{{\mathcal S_\mathcal T}}}$ [$-\varepsilon W_{j} |_{\widehat{{\mathcal S_\mathcal T}}}$], for a given ${j}\in \{1,\dots,n\}$,
the spectrum of the perturbed matrix $T_{j}^{\varepsilon}$ [$T_{j}^{-\varepsilon}$] contains the eigenvalue
\[
\lambda_{{j}}^{\varepsilon}=\delta+\frac{\varepsilon}
{n\sqrt{\frac{1}{n}+\frac{2}{n-1} \cos^2 \frac{{j}\pi}{n+1}}}+ 
2\left(\sigma+\frac{\varepsilon\cos \frac{{j}\pi }{n+1}}{(n-1)
\sqrt{\frac{1}{n}+\frac{2}{n-1} \cos^2 \frac{{j}\pi}{n+1}}}\right)\,\cos \frac{j\pi }{n+1}
\]
\[
[\lambda_{{j}}^{-\varepsilon}=\delta-\frac{\varepsilon}
{n\sqrt{\frac{1}{n}+\frac{2}{n-1} \cos^2 \frac{{j}\pi}{n+1}}}+ 
2\left(\sigma-\frac{\varepsilon\cos \frac{{j}\pi }{n+1}}{(n-1)
\sqrt{\frac{1}{n}+\frac{2}{n-1} \cos^2 \frac{{j}\pi}{n+1}}}\right)\,\cos \frac{j\pi }{n+1}].
\]

Straightforwardly, the ${\mathcal S_\mathcal T}$-structured $\varepsilon$-pseudospectrum, 
for $\varepsilon$ small enough, is given by the union of the real intervals 
$[\lambda_{h}^{-\varepsilon},\lambda_{h}^{+\varepsilon} ]$ of width 
$2\,\kappa_{{\mathcal S_\mathcal T}}(\lambda_h)\varepsilon$, for $h=1:n$.
\end{remark}

Let us turn to the shifted skew-symmetric case. 
\begin{proposition}
All the eigenvalues of a shifted skew-symmetric tridiagonal Toeplitz matrix 
$T\in\mathbb{R}^{n \times n}$ have the same condition number 
$$\kappa_{{\mathcal A_\mathcal T}}(\lambda_h)=\frac{1}{\sqrt n}$$ with respect to any 
structured perturbation that respects the shifted skew-symmetry. 
\end{proposition}

\begin{proof}
Odd eigenvector components of real shifted skew-symmetric tridiagonal Toeplitz matrices 
are purely imaginary numbers. Hence, the Wilkinson perturbation associated with $\lambda_h$ 
is symmetric. By using the same notation as in Remark \ref{rem_sym}, we obtain
\[
\sigma_h=\bar{\tau}_h=\frac{\text{sgn}(\tau) i}{n-1} 
\cos \frac{h\pi}{n+1};\qquad \delta_h=\frac{1}{n}.
\]
Thus, the orthogonal projection of $W_h$ in the subspace  of real shifted skew-symmetric 
tridiagonal Toeplitz matrices is the matrix $\frac{1}{n}I$. Its Frobenius norm 
$\frac{1}{\sqrt n}$ gives the structured condition number 
$\kappa_{{\mathcal A_\mathcal T}}(\lambda_{h} )=\|W_h |_{{\mathcal A_\mathcal T}}\|_F$. 
\end{proof}

\begin{remark}
Perturbing the real shifted skew-symmetric tridiagonal matrix 
$T=(n; \sigma,\delta,-\sigma)$ by $\pm\varepsilon W_h |_{\widehat{{\mathcal A_\mathcal T}}}$, 
where $W_h|_{\widehat{{\mathcal A_\mathcal T}}}$ is the $\mathcal A_\mathcal T$-structured 
unit-norm analogue of the Wilkinson perturbation, gives the pseudoeigenvalues 
$\lambda_{h}^{\pm\varepsilon} =\delta\pm \frac{\varepsilon}{\sqrt n} +2i|\sigma|\,
\cos \frac{h\pi }{n+1}$ for $h=1:n$. 
\end{remark}
 
We conclude this section by noticing that Proposition \ref{eqeig} can be extended to real
symmetric tridiagonal Toeplitz-type matrices. We have the following result.

\begin{proposition}\label{eqeigTt}
Any real symmetric tridiagonal Toeplitz-type matrix of a fixed order $n$ of the types 
considered in Table \ref{table1} has the same right and left eigenvectors.
\end{proposition}

\begin{proof}
One has 
\[
\begin{array}{rclll}
x_{h,k}&=&\sin\frac{2hk\pi}{2n+1},&\quad 
\alpha=0,&\quad\beta=\sigma;\\
x_{h,k}&=&\sin\frac{h(2k-1)\pi}{2n+1},&\quad
\alpha=\sigma,&\quad\beta=0;\\
x_{h,k}&=&\sin\frac{(2h-1)k\pi}{2n+1},&\quad
\alpha=0,&\quad\beta=-\sigma;\\
x_{h,k}&=&\cos\frac{(2h-1)(2k-1)\pi}{2(2n+1)},&
\quad \alpha=-\sigma,&\quad\beta=0;\\
x_{h,k}&=&\sin\frac{(2h-1)(2k-1)\pi}{4n},&
\quad \alpha=\sigma,&\quad\beta=-\sigma;\\
x_{h,k}&=&\cos\frac{(2h-1)(2k-1)\pi}{4n},&\quad
\alpha=-\sigma,&\quad\beta=\sigma;\\
x_{h,k}&=&\sin\frac{h(2k-1)\pi}{2n},&\quad 
\alpha=\sigma,&\quad\beta=\sigma;\\
x_{h,k}&=&\cos\frac{(h-1)(2k-1)\pi}{2n},&\quad
\alpha=-\sigma,&\quad\beta=-\sigma
\end{array}
\]
for $k=1:n$.
\end{proof}

\section{Applications}\label{sec5}
This section discusses how the theory developed in the previous sections can be applied
to approximate the eigenvalues or accurately evaluate the spectral factorization of
certain matrices.

\subsection{Approximation of the spectrum of a real symmetric tridiagonal matrix}
Let $A_n\in\mathbb{R}^{n \times n}$ be a symmetric tridiagonal matrix. Denote the $j$th
subdiagonal entry of $A_n$ by $\sigma_j$, $j=1:n-1$, and let $\delta_j$ be the $j$th 
diagonal entry, $j=1:n$. The matrix $A_n$ may, for instance, have been determined by 
carrying out $n$ steps of the symmetric Lanczos algorithm applied to a large symmetric 
matrix $A$; see, e.g., \cite{GVL} for a discussion on this algorithm. 

Let  $T:=A_n|_{\mathcal{T}}$ be the orthogonal projection of $A_n$ in the subspace 
${\mathcal T}$ of tridiagonal Toeplitz matrices. We are interested in the matrix $T$ 
because its eigenvalues are known in closed form and can be used to estimate the 
eigenvalues of $A_n$.

\begin{proposition}
$T$ is a real symmetric tridiagonal Toeplitz matrix. 
\end{proposition}
\begin{proof}
The proof is straightforward, because both the subdiagonal and superdiagonal entries of $T$
are equal to $\frac{\sum_{j=1}^{n-1}\sigma_j}{n-1}$.
\end{proof}

\begin{proposition}\label{spT}
If the trace of $A_n$ vanishes, then the spectrum of $T$ is real and symmetric with 
respect to the origin. Moreover, if $n$ is odd, then $T$ is singular. For $n$ even, 
\[
\kappa_2(T)=\frac{\cos \frac{\pi}{n+1}}{\cos \frac{n\pi}{2(n+1)}}.
\]
\end{proposition}

\begin{proof}
The diagonal entries of $T$, given by $\delta=\frac{\sum_{j=1}^{n}\delta_j}{n}$, vanish. 
Therefore the spectrum $\{\lambda_j\}_{j=1}^n$ of $T$ is symmetric with respect to the 
origin. If $n$ is odd, zero is an eigenvalue; otherwise, if $n$ is even, one has 
$\kappa_2(T)=\lambda_1/\lambda_{\frac{n}{2}}$, where the eigenvalues are defined by
\eqref{lamhT} with $\tau=\sigma$. This concludes the proof.
\end{proof}


We have that $T$ coincides with $A_n$ if and only if $A_n$ is a Toeplitz matrix. Thus,
trivially, if $A_n$ is a scalar, then $T$ coincides with $A_n$. Moreover, the following 
inequality holds.

\begin{proposition}\label{prpsum}
Let $\lambda_1(A_n)\geq\ldots\geq\lambda_n(A_n)$ denote the eigenvalues of $A_n$ in 
decreasing order and let $\lambda_i$ be the eigenvalues of $T$ given by \eqref{lamhT} with
$\tau=\sigma$. Then the average of the squared distances between the eigenvalues of $A_n$ 
and $T$ satisfies
\[
\frac{1}{n}\sum_{i=1}^{n}(\lambda_i(A_n)-\lambda_i)^2\leq
\frac{1}{n}\|A_n-T\|_F^2.
\]
\end{proposition}

\begin{proof}
Let $A,B\in{\R}^{n\times n}$ be symmetric matrices. Denote by 
$\lambda_{\downarrow}(M)$ [$\lambda_{\uparrow}(M)$] the vector whose entries are the 
eigenvalues of a symmetric matrix $M$ sorted in decreasing [increasing] order. Then 
\[
\|\lambda_{\downarrow}(A)-\lambda_{\downarrow}(B)\|\leq\|A-B\|_F\leq 
\|\lambda_{\downarrow}(A)-\lambda_{\uparrow}(B)\|;
\]
see, e.g., \cite{RB86}. This shows the proposition.
\end{proof}

\begin{remark}
Notice that $A_n$ being symmetric positive definite does not guarantee that $T$ is 
positive definite. Indeed, $T$ is positive definite if and only if 
\[
\frac{\sum_{j=1}^{n}\delta_j}{n}>
2\frac{\sum_{j=1}^{n-1}\sigma_j}{n-1} \cos \frac{\pi}{n+1}.
\]
\end{remark}

Let $T=(n;\sigma,\delta,\tau)$ be symmetric and define the Toeplitz-type matrix
$A_n:=T_{\alpha,\beta}$, where $\alpha=\pm\sqrt{\sigma\tau}$ and 
$\beta=\mp\sqrt{\sigma\tau}$. The eigenvalues of the matrix $A_n$ are symmetric with 
respect to $\delta$; expressions for the eigenvalues are provided in the fifth and sixth 
rows of Table \ref{table1}. It is easy to show that $T$ is the closest tridiagonal 
Toeplitz matrix to $A_n$ in the Frobenius norm. Moreover, if $\delta=0$, then the 
eigenvalues of $A_n$ are symmetric with respect to the origin and $A_n$ has null trace so
that, due to Proposition \ref{spT}, the spectrum of $T$ is symmetric with respect to the 
origin. 

We illustrate Proposition \ref{prpsum} with an example. Let the matrix 
$A_n=[a_{i,j}]\in\R^{n\times n}$ differ from the symmetric Toeplitz matrix 
$T=(n,\sigma,\delta,\sigma)$ only in the entry $a_{2,2}$. Then the proposition shows that
\[
\frac{1}{n}\sum_{i=1}^{n}(\lambda_i(A_n)-\lambda_i)^2\leq
\frac{1}{n}|a_{2,2}-\delta|^2. 
\]
In particular, the sum in the left-hand side converges to zero as $n$ increases. Hence,
the spectrum of $T$ furnishes an accurate approximation of the spectrum of $A_n$ when
$n$ is large.

\subsection{Accurate computation of the spectrum of nonsymmetric nearly tridiagonal Toeplitz 
matrices}
Let the tridiagonal Toeplitz matrix $T=(n;\delta,\sigma,\tau)$ be nonsymmetric. It has 
the spectral factorization 
\begin{equation}\label{specfact}
T=X\Lambda X^{-1},
\end{equation}
where $X\in{\mathbb C}^{n\times n}$ is the eigenvector matrix whose columns are given by
\eqref{xhk} and the entries of the matrix $\Lambda={\rm diag}[\lambda_1,\lambda_2,\ldots,\lambda_n]$
are the eigenvalues given by \eqref{lamhT}. 

When $T\in{\mathbb R}^{n\times n}$ is far from symmetric, then the MATLAB function {\sf eig} is only able to compute 
the spectral factorization \eqref{specfact} with reduced accuracy. For instance, consider 
the matrix $T=(25;1,0,0.01)$. The eigenvalues of $T$ are given by
\begin{equation}\label{exspect1}
\lambda_h=0.2\cos\frac{h\pi}{26},\quad h=1:25,
\end{equation}
while many of the eigenvalues determined by the function {\sf eig} have a significant
imaginary part; see Figure \ref{fig1}. 

\begin{figure}
\centerline{
\includegraphics[scale=0.55,trim= 0mm 0.01mm 0mm 0mm]{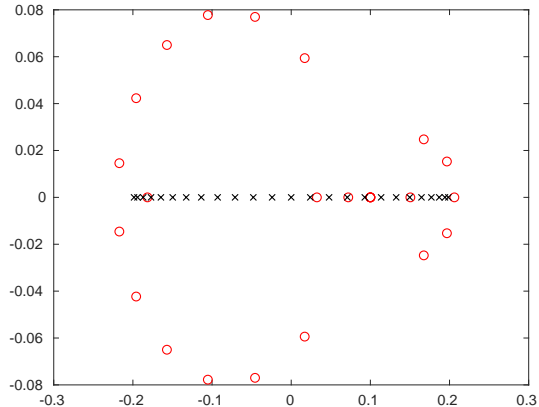}
}
\caption{Exact eigenvalues \eqref{exspect1} of $T=(25;1,0,0.01)$ (marked with black {\rm x}) 
and the approximate eigenvalues computed by the function {\sf eig} (marked with red 
\textcolor{red}{\rm o}).}\label{fig1}
\end{figure}

Also the spectrum of other nonsymmetric matrices can be difficult to compute accurately by
the function {\sf eig}. When the matrix of interest, $A_n\in{\mathbb R}^{n\times n}$, is 
close to a Toeplitz matrix $T=(n;\sigma,\delta,\tau)$, the spectral factorization 
\eqref{specfact} may be used to determine a more accurate spectral factorization of $A_n$ 
than can be computed with {\sf eig} in the following manner:
\begin{enumerate}
\item Determine the tridiagonal Toeplitz matrix $T$ closest to $A_n$ in the Frobenius 
norm.
\item 
Determine the spectral factorization \eqref{specfact} of $T$ by using (\ref{lamhT}) and 
(\ref{xhk}).
\item Evaluate the matrix $B=X^{-1}A_nX=\Lambda+X^{-1}(A_n-T)X$. If $T$ is close to
$A_n$, then this matrix is closer to a symmetric matrix than $A_n$.
\item Compute the spectral factorization $B=YDY^{-1}$ by using the MATLAB function 
{\sf eig}. Thus, $Y$ is the eigenvector matrix of $B$, and $D$ is a diagonal matrix, whose
nontrivial entries are the eigenvalues. Typically, the matrix $Y$ is fairly well 
conditioned and can be computed by the function {\sf eig} with quite high accuracy.
The matrix $Z=XY$ is (an approximation of) the eigenvector matrix of $A_n$. 
\end{enumerate}

We illustrate the computations outlined with an example. Let $T=(25;1,0,0.01)$, and
let $A_n=T_{\alpha,\beta}\in\mathbb{R}^{n \times n}$ be a tridiagonal Toeplitz-type matrix
\eqref{Talbeta} obtained from $T$ with $\alpha=0.1$ and $\beta=-\alpha$. The eigenvalues of $A_n$ are 
real and symmetric with respect to the origin; their formulas are shown in the fifth
and sixth rows of Table \ref{table1}. The eigenvectors of $A_n$ are described in Section 
\ref{sec1}. Hence, it is straightforward to assess the accuracy of the computational 
method described. It is easy to see that $T$ is the closest tridiagonal Toeplitz matrix to
$A_n$. Its eigenvalues and eigenvectors are given by (\ref{lamhT}) and (\ref{xhk}).

Figure \ref{fig4} displays the spectrum of $A_n$ computed by using the relevant 
formulas of Table \ref{table1} (marked with black {\rm +}), and approximations of the spectrum 
computed by the MATLAB function {\sf eig} (marked with red \textcolor{red}{\rm o}) and the 
procedure described above (marked with blue \textcolor{blue}{\rm x}). The eigenvalues determined in the
latter manner cannot be distinguished from the exact ones in Figure \ref{fig4}, while 
some of the approximate eigenvalues computed by {\sf eig} applied to $A_n$ can be seen to 
have large imaginary components. The maximum pairwise difference of the exact eigenvalues 
and the eigenvalues computed by the MATLAB function {\sf eig}, ordered in the same manner, 
is $4.3\cdot 10^{-1}$, while the maximum pairwise difference of the exact eigenvalues and 
the eigenvalues computed by our approach described above only is $3.3\cdot 10^{-8}$. Thus, 
the approximation of a tridiagonal matrix by the closest Toeplitz matrix and using the 
spectral factorization of the latter may yield a more accurate spectral factorization than 
the one determined by the MATLAB function {\sf eig}.

\begin{figure}
\centerline{
\includegraphics[scale=0.55,trim= 0mm 0.01mm 0mm 0mm]{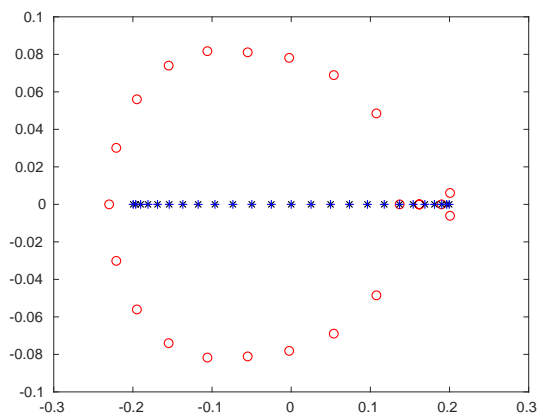}
}
\caption{Exact eigenvalues of $A_n$ (marked with black {\rm x}), the approximate eigenvalues computed by the 
function {\sf eig} (marked with red \textcolor{red}{\rm o}), and eigenvalues computed by the algorithm outlined above 
(marked with blue \textcolor{blue}{\rm +}).}\label{fig4}
\end{figure}

\section{Conclusions}\label{sec6}
The paper discusses the sensitivity of eigenvectors of tridiagonal Toeplitz matrices
under general and structured perturbations. The eigenvectors are found to be quite
sensitive to perturbations when the Toeplitz matrix is far from normal, but the
eigenvectors are insensitive to structured perturbation when the Toeplitz matrix has
additional structure, such as being real symmetric. Our analysis suggests a novel method
for computing the spectral factorization of a general nonsymmetric tridiagonal matrix.



\newpage

\begin{thebibliography}{99}
\bibitem{BF}
N. Bebiano and S. Furtado, Structured distance to normality of tridiagonal matrices,
Linear Algebra and its Applications, 552 (2018), pp. 239--255.
\bibitem{RB86}
R. Bhatia, The distance between the eigenvalues of symmetric matrices, Proceedings of the
American Mathematical Society, 96 (1986), pp. 41--42.
\bibitem{BBGM15}
J. M. Bogoya, A. B\"ottcher, S. M. Grudsky, and E. A. Maximenko, 
Maximum norm versions of the Szeg\H o and Avram-Parter theorems for Toeplitz matrices,
Journal of Approximation Theory, 196 (2015), pp. 79--100.
\bibitem{BBGM17}
A. B\"ottcher, J. M. Bogoya, S. M. Grudsky, and E. A. Maximenko,
Asymptotic of eigenvalues and eigenvectors of Toeplitz matrices,
Sbornik: Mathematics, 208 (2017), pp. 1578--1601.
\bibitem{BG05}
A. B\"ottcher, and S. M. Grudsky,
Spectral Properties of Banded Toeplitz Matrices. SIAM, Philadelphia, 2005.
\bibitem{BGN}
P. Butt\`a, N. Guglielmi, and S. Noschese, Computing the structured pseudospectrum of a 
Toeplitz matrix and its extreme points, SIAM Journal on Matrix Analysis and Applications, 
33 (2012) pp.  1300--1319.
\bibitem{DL98}
F. Diele and L. Lopez, The use of the factorization of five-diagonal matrices by 
tridiagonal Toeplitz matrices, Applied Mathematics Letters, 11 (1998), pp. 61--69. 
\bibitem{FGHLW74}
D. Fischer, G. Golub, O. Hald, C. Leiva, and O. Widlund, On Fourier-Toeplitz methods for 
separable elliptic problems, Mathematics of Computation, 28 (1974), pp. 349--368.
\bibitem{GVL}
G. H. Golub and C. F. Van Loan, Matrix Computations, 4th ed., Johns Hopkins University 
Press, 2013.
\bibitem{H98}
P. C. Hansen, Rank-Deficient and Discrete Ill-Posed Problems, SIAM, Philadelphia, 1998.
\bibitem{KKT}
M. Karow, D. Kressner, and F. Tisseur, Structured eigenvalue condition numbers,
SIAM Journal on Matrix Analysis and Applications, 28 (2006), pp. 1052--1068.
\bibitem{LP10}
A. Luati and T. Proietti, On the spectral properties of matrices associated with trend
filters, Econometric Theory, 26 (2010), pp. 1247--1261.
\bibitem{NP1}
S. Noschese and L. Pasquini, Eigenvalue condition numbers: zero-structured
versus traditional, Journal of Computational and Applied Mathematics, 185 (2006), pp. 174--189.
\bibitem{NP2}
S. Noschese and  L. Pasquini, Eigenvalue patterned condition numbers: Toeplitz
and Hankel cases, Journal of Computational and Applied Mathematics, 206 (2007), pp. 615--624.
\bibitem{NPR1}
S. Noschese, L. Pasquini, and L. Reichel, The structured distance to normality
of an irreducible real tridiagonal matrix, Electronic Transactions on Numererical Analysis,
28 (2007), pp. 65--77.
\bibitem{NPR}
S. Noschese, L. Pasquini, and L. Reichel, Tridiagonal Toeplitz matrices: Properties and 
novel applications, Numerical Linear Algebra with Applications, 20 (2013), pp. 302--326.
\bibitem{NR}
S. Noschese and L. Reichel, The structured distance to normality of banded Toeplitz 
matrices, BIT Numerical Mathematics, 49 (2009), pp. 629--640.
\bibitem{NR2}
S. Noschese and L. Reichel, Approximated structured pseudospectra, Numerical Linear 
Algebra with Applications, 24 (2017), e2082.
\bibitem{NR19}
S. Noschese and L. Reichel, Computing unstructured and structured polynomial pseudospectrum 
approximations, Journal of Computational and Applied Mathematics, 350 (2019), pp. 57--68.
\bibitem{P98}
B. N. Parlett, The Symmetric Eigenvalue Problem, SIAM, Philadelphia, 1998.
\bibitem{RT}
L. Reichel and L. N. Trefethen, Eigenvalues and pseudo-eigenvalues of Toeplitz
matrices, Linear Algebra and its Applications, 162-164 (1992), pp. 153--185.
\bibitem{RY09}
L. Reichel and Q. Ye, Simple square smoothing regularization operators, Electronic 
Transactions on Numerical Analysis, 33 (2009), pp. 63--83.
\bibitem{Sm}
G. D. Smith, Numerical Solution of Partial Differential Equations, 2nd ed., 
Clarendon Press, Oxford, 1978.
\bibitem{St01}
G. W. Stewart, Matrix Algorithms, Volume II: Eigensystems. SIAM, Philadelphia, 2001.
\bibitem{TE}
L. N. Trefethen and M. Embree, Spectra and Pseudospectra, Princeton University 
Press, Princeton, 2005.
\bibitem{Wi}
J. H. Wilkinson, Sensitivity of eigenvalues II, Utilitas Mathematica, 30 (1986), pp. 243--286.
\bibitem{Y}
W.-C. Yueh, Eigenvalues of several tridiagonal matrices, Applied Mathematics E-notes, 5
(2005), pp. 66--74.
\bibitem{YC08}
W.-C. Yueh and S. S. Cheng, Explicit eigenvalues and inverses of tridiagonal Toeplitz 
matrices with four perturbed corners, ANZIAM Journal, 49 (2008), pp. 361--387.
\end{thebibliography}
\end{document}